\documentclass[10pt]{iopart}

\usepackage{iopams}
\usepackage{algorithm}
\usepackage{algorithmic}
\newtheorem{theorem}{Theorem}
\newtheorem{proposition}{Proposition}

\usepackage{epsfig}
\usepackage{tabularx}
\usepackage{graphicx}
\usepackage{booktabs}
\usepackage[dvipsnames]{xcolor}

\usepackage[colorlinks]{hyperref}
\hypersetup{citecolor=red}
\hypersetup{linkcolor=green}
\hypersetup{urlcolor=ProcessBlue}

\newenvironment{proof}{{\it Proof.}}{\hfill$\square$\vspace{0.3cm}\\}

\usepackage{iopams} 
\expandafter\let\csname equation*\endcsname\relax
\expandafter\let\csname endequation*\endcsname\relax
\usepackage{amsmath}

\newcommand{\minimize}{\mathop{\vphantom{\sum}\mathchoice
  {\vcenter{\hbox{ minimize}}}
  {\vcenter{\hbox{ minimize}}}{\mathrm{minimize}}{\mathrm{minimize}}}\displaylimits}
\begin{document}
\title[]{Deep Unfolding of a Proximal Interior Point Method for Image Restoration\footnote[1]{Contact author: M.-C. Corbineau, marie-caroline.corbineau@centralesupelec.fr.}}

\author{C. Bertocchi$^1$,  E. Chouzenoux$^{2}$, M.-C. Corbineau$^2$, J.-C. Pesquet$^2$, and 
M. Prato$^1$}

\address{$^1$ Universit\`a di Modena e Reggio Emilia, Modena, Italy}
\address{$^2$ CVN, CentraleSup\'elec, INRIA Saclay, Universit\'e Paris-Saclay, Gif-Sur-Yvette, France}
\ead{carla.bertocchi@unimore.it, emilie.chouzenoux@centralesupelec.fr, marie-caroline.corbineau@centralesupelec.fr, jean-christophe.pesquet@centralesupelec.fr, marco.prato@unimore.it}

\begin{abstract}
Variational methods are widely applied to ill-posed inverse problems for they have the ability to embed prior knowledge about the solution. However, the level of performance of these methods significantly depends on a set of parameters, which can be estimated through computationally expensive and time-consuming methods. In contrast, deep learning offers very generic and efficient architectures, at the expense of explainability, since it is often used as a black-box, without any fine control over its output. 
Deep unfolding provides a convenient approach to combine variational-based and deep learning approaches. 
Starting from a variational formulation for image restoration, we develop iRestNet, a neural network architecture obtained by unfolding a proximal interior point algorithm. Hard constraints, encoding desirable properties for the restored image, are incorporated into the network thanks to a logarithmic barrier, while the barrier parameter, the stepsize, and the penalization weight are learned by the network. We derive explicit expressions for the gradient of the proximity operator for various choices of constraints, which allows training iRestNet with gradient descent and backpropagation. In addition, we provide theoretical results regarding the stability of the network for a common inverse problem example. Numerical experiments on image deblurring problems show that the proposed approach compares favorably with both state-of-the-art variational and machine learning methods in terms of image quality.
\end{abstract}

\vspace{2pc}

%\submitto{\IP Special Issue on Variational Methods and Effective Algorithms for Imaging and Vision}

\maketitle

\section{Introduction}
\label{sec:intro}
In this work we focus on inverse problems related to the following model:
\begin{equation}
y=\mathcal{D}(H\overline{x}),
\label{pb:inverse_pb}
\end{equation}
where $y\in\mathbb{R}^m$ is the observed data, $\overline{x}\in\mathbb{R}^n$ is the sought signal or image, $H\in\mathbb{R}^{m\times n}$ is the observation operator, which is assumed linear, and $\mathcal{D}$ is the noise perturbation operator. The linear operator $H$ is assumed to be known from a physical model or prior identification step \cite{Lagendijk2005, xu2010two}. In this context, both variational and deep learning approaches provide efficient methods for delivering an estimate of $\overline{x}$, while offering different benefits and drawbacks, which are discussed hereafter.
\\\\ 
In order to find an appropriate solution to an ill-posed inverse problem like (\ref{pb:inverse_pb}), variational methods incorporate prior information on the sought variable $\overline{x}$, through constraints or regularization functions, such as the total variation and its various extensions \cite{Aujol2009firstorder} or sparsity-promoting functions \cite{Pustelnik16}. This leads to the following minimization problem,
\begin{equation}
\min_{x\in\mathcal{C}} ~ f(Hx,y)+\lambda\mathcal{R}(x)
\label{pb:var_pb2}
\end{equation}
%\begin{equation}
%\min_{x\in\mathcal{C}} ~ f(Hx,y)+\mathcal{R}(x)
%\label{pb:var_pb}
%\end{equation}
where $f:\mathbb{R}^m\times\mathbb{R}^m\rightarrow\mathbb{R}$ is a data-fidelity function, which is convex with regards to its first variable and which is related to the degradation model, $\mathcal{R}:\mathbb{R}^n\rightarrow\mathbb{R}$ is a convex regularization function, $\lambda\in \ ]0,+\infty[$ is a regularization parameter and $\mathcal{C}$ is a subset of $\mathbb{R}^n$. Although useful, this approach is sometimes limited by its complexity: solving (\ref{pb:var_pb2}) may require advanced algorithms that may be too slow for real-time applications. In addition, $\lambda$ is a parameter that needs to be set and $\mathcal{R}$ is usually parametrized by one or several parameters, whose optimal choice may strongly depend on the data at hand. These parameters are often tuned manually or computed using, for instance, cross validation, the discrepancy principle \cite{scherzer1993use}, or methods based on Stein unbiased risk estimates  (SURE) \cite{deledalle2014stein}. However, these methods are often time-consuming and their success is not always guaranteed. Furthermore, despite numerous efforts in designing sophisticated models, the solution to (\ref{pb:var_pb2}) could be further away from $\overline{x}$ than an intermediate iterate produced by a given algorithm used for solving (\ref{pb:var_pb2}). Such phenomenon justifies the development of early stopping methods, where the iterative procedure is stopped before convergence \cite{yao2007early}. Finding the optimal stopping time depends on the algorithm and requires the use of an oracle such as SURE, which may explain why these techniques are currently restricted to relatively simple cost functions.
\\\\
Deep Neural Networks (DNNs), and in particular Convolutional Neural Networks (CNNs), provide good performance for various applications related to inverse problems, such as denoising \cite{zhang2017beyond}, non-blind and blind deblurring \cite{xu2014deep,schuler2013machine,schuler2016learning}, super-resolution \cite{ledig2017photo}, or CT reconstruction \cite{jin2017deep}. As detailed in \cite{mccann2017convolutional}, DNNs for inverse problems are very often preceded by a pre-processing step. Indeed, a rough estimation of $\overline{x}$ can be found by using the inverse or pseudoinverse of $H$. The latter tends, however, to strongly amplify noise. Hence, in this context, DNNs are used as denoisers and artifact-removers. However, since prior knowledge about its output can hardly be incorporated into a DNN, which in most of the cases is viewed as a black-box, the explainability and reliability \cite{szegedy2013intriguing} of such methods could be questioned. Furthermore, the pre-processing step, in itself, can include a penalty, thus amounting to solving a problem of the form (\ref{pb:var_pb2}), where the regularization weight strongly depends on the noise level, \textit{e.g.} \cite{schuler2013machine,boublil2015spatially}. One straightforward way to combine the benefits of both variational-based methods and DNNs is to unfold an iterative method and untie the parameters of both the model and the algorithm across the layers of the network \cite{hershey2014deep}. Interestingly, the fact that this approach makes use of a limited number of layers can be viewed as an analogue of early stopping methods. It is however worth mentioning that, in unfolded algorithms, the number of iterations (i.e., layers) is tuned during the off-line training step and is then fixed for all test images, which differs from early stopping strategies where the iteration number usually differs for
each processed image.
\\\\
In this paper, we propose a novel neural network architecture called iRestNet, which is obtained by unfolding a proximal interior point algorithm over a finite number of iterations. One key feature of this algorithm is that it produces only feasible iterates thanks to a logarithmic barrier. This barrier enables prior knowledge to be directly incorporated into iRestNet and, as opposed to a projection onto $\mathcal{C}$, it allows differentiation and gradient backpropagation throughout the network. Hence, gradient descent can be used for training. The stepsize, barrier parameter, and regularization weight are untied across the network and learned for each layer. Thus, once the network has been trained, its application on test images requires only a short execution time per image without any parameter search, as opposed to traditional variational methods.
\\\\
Related works apply deep unfolding to probabilistic models, such as Markov random fields \cite{hershey2014deep}, topic models \cite{chien2018deep}, and to different algorithms like primal-dual solvers \cite{wang2016proximal} or the proximal gradient method \cite{mardani2017recurrent,diamond2017unrolled}. Classic optimization algorithms can be unfolded to perform many different tasks in image processing. For instance, FISTA and ISTA can be unfolded to perform sparse coding~\cite{gregor2010learning,kamilov2016learning}, while the same ISTA and ADMM can be unfolded for image compressive sensing~\cite{zhang2018ista,sun2016deep}. However, in the aforementioned works, some functions and operators are learned, which weakens the link between the resulting network and the original algorithm. Deep unfolding is also used to learn shrinkage functions, which can be viewed as proximity operators of sparsity-promoting functions \cite{schmidt2014shrinkage,sun2015color}, or to optimize hyperparameters in nonlinear reaction diffusion models \cite{chen2017trainable}. 
Several recent works consider replacing handcrafted algorithms by learned iterative methods \cite{andrychowicz2016learning,li2016learning}. 
In these approaches, the goal is to find the minimizer of a given objective function, whereas in the proposed method, the architecture is inspired by an optimization strategy applied to the minimization of an objective function but a better indicator of perceptual quality is optimized during the training step.
Only a few works so far have considered combining interior point methods (IPMs) with deep learning. Every layer of the network from \cite{Amos2018} solves a small quadratic problem using an IPM, while in \cite{Trafalis97}, hard constraints are enforced on weights by using the logarithmic barrier function during training. More recently, an interior point strategy was used to design a recurrent network, whose purpose is to solve a specific convex constrained problem \cite{Krasopoulos2014}. 
In our case, however, we have two distinct objective functions. The first one leads to a constrained problem from which the proposed architecture is inspired, while the second one is used during training as a loss function. It is worth noting that the output of the trained network is not necessarily a minimizer of the first objective. Moreover, the second objective could not be substituted to the first one since it requires the knowledge of the ground-truth, which is available for training time but not in testing conditions. In addition, iRestNet appears to have more flexibility since the regularization weight can vary among layers.
%%%  
\\\\
To the best of our knowledge, this paper presents the first architecture corresponding to a deep unfolded version of an interior point algorithm with untied stepsize and regularization parameter. As opposed to other unfolding methods like \cite{mardani2017recurrent,diamond2017unrolled}, the proximity operator and the regularization term are kept explicit, which establishes a direct relation between the original algorithm and the network. Other contributions of this work include the expression of the required proximity operator, and of its corresponding gradient, for three standard variational formulations, along with numerical experiments demonstrating the benefit of using the proposed approach over other machine learning and variational methods for image deblurring. 
\\\\
This paper is organized as follows: in Section~\ref{sec:IPM}, we describe the proximal interior point optimization method which is at the core of our approach, and we provide the proximity operator of the barrier for three useful cases in Section~\ref{sec:prox_computation}. In Section~\ref{sec:net}, we present the proposed neural network architecture and its associated backpropagation method. In Section~\ref{sec:network_stability}, we conduct a stability analysis of the proposed network when the data fidelity term and the regularization function are quadratic. 
Section \ref{sec:exp} is dedicated to numerical experiments and comparison to state-of-the-art methods for image deblurring; finally, some conclusions are drawn in Section \ref{sec:conclu}.

\section{Proximal interior point algorithm}
\label{sec:IPM}
\subsection{Variational formulation and notation}

As detailed in Section~\ref{sec:intro}, the sought image $\overline{x}$ can be classically approximated by the minimizer of a penalized cost function expressed as the sum of a data-fitting term, which measures the fidelity of the solution to the observation model (\ref{pb:inverse_pb}), and a regularization term, which is introduced so as to avoid meaningless solutions and improve stability relative to noise. This leads to problem (\ref{pb:var_pb2}) with $\lambda \in \ ]0,+\infty[$ a regularization parameter. 
For every $q\in\mathbb{N}$, let $\Gamma_0(\mathbb{R}^q)$ denote the set of functions which take values in $\mathbb{R}\cup\{+\infty\}$ and are proper, convex, lower semicontinuous on $\mathbb{R}^q$. In the remaining of the paper, we will assume that, for every $y\in\mathbb{R}^m$, $f(\cdot,y)\in\Gamma_0(\mathbb{R}^m)$ is a twice-differentiable data-fidelity term, and $\mathcal{R}\in\Gamma_0(\mathbb{R}^n)$ is a twice-differentiable regularization function. Note that such assumption is necessary to define the derivative steps involved in the backpropagation procedure for the training of our network. The feasible set $\mathcal{C}$ is defined by $p$ inequality constraints which enforce the fulfillment of some properties that are expected to be satisfied a priori by the image:
\begin{equation}
\mathcal{C}=\{x\in\mathbb{R}^n~|~(\forall i\in\{1,\ldots,p\})~c_i(x)\geq 0\},
\label{def:feasible_set}
\end{equation}
where, for every $i\in\{1,\ldots,p\}$, $-c_i\in\Gamma_0(\mathbb{R}^n)$. The strict interior of the feasible domain, $\rm{int}\mathcal{C}$, is equal to 
\begin{equation}
\mathrm{int}\mathcal{C}=\{x\in\mathbb{R}^n~|~(\forall i\in\{1,\ldots,p\})~c_i(x)> 0\},
\end{equation}
and it is assumed to be nonempty. 
Inequality constraints are frequently used in image processing. For instance, they can be derived from the underlying geometry of the problem, like in \cite{harizanov2013epigraphical}, in the context of Poisson-noise denoising. We can also mention the work in \cite{musse2001topology}, where inequality constraints are used in a problem of deformable image matching to ensure that the estimated image deformation is injective and preserves the topology. Constraints can also serve to enforce some a priori knowledge about the solution, as in the image segmentation approach in \cite{klodt2011convex}, where bound constraints are imposed on the segmented areas and their barycenters.
Finally, we will assume that either $f(H\cdot,y)+\lambda\mathcal{R}$ is coercive, or $\mathcal{C}$ is bounded. Then the existence of solutions for (\ref{pb:var_pb2}) is guaranteed. It is worthy to emphasize that a large class of penalized formulations encountered in the literature of image restoration fulfills the above requirements, see e.g. \cite{durand2006stability} and references therein. 
\\\\
Let us introduce additional notations, which will be useful in the rest of the paper. First, for every $g\in\Gamma_0(\mathbb{R}^n)$, $\gamma\in \ ]0,+\infty[$, and $x\in\mathbb{R}^n$, the proximity operator \cite{bauschke2017convex} of $\gamma g$ at $x$ is uniquely defined as 
\begin{equation}
\mathrm{prox}_{\gamma g}(x)=\underset{u\in\mathbb{R}^n}{\mathrm{argmin}} \ \frac{1}{2}\|x-u\|^2+\gamma g(u).
\end{equation}
Finally, for all $(x,y,\lambda) \in \mathbb{R}^n\times\mathbb{R}^m\times\ ]0,+\infty[$, we define
\begin{equation}
h(x,y,\lambda)=f(Hx,y)+\lambda\mathcal{R}(x),
\end{equation} 
and
\begin{equation}
\nabla_1h(x,y,\lambda)=H^\top\nabla_1f(Hx,y)+\lambda\nabla \mathcal{R}(x),
\end{equation}
where $\nabla_1f$ is the partial gradient of $f$ with respect to its first variable.

\subsection{Interior point approaches}

In general, problem~(\ref{pb:var_pb2}) does not have a closed-form solution on account of the inequality constraints, even for simple regularizations, hence an iterative solver must be used. Several resolution approaches are available, either based on projected gradient strategies \cite{iusem03pgd,bonettini2015new}, ADMM \cite{boyd2011admm}, primal-dual schemes \cite{komodakis2015playing}, or interior point techniques \cite{bonettini2009nonneg}. Standard interior point methods require to invert several $n \times n$ linear systems, which leads to a high computational complexity for large scale problems. Nonetheless, it has recently been shown that combining the interior point framework with a proximal forward--backward strategy \cite{combettes2005fb,combettes2010book} leads to very competitive solvers for inverse problems \cite{chouzenoux2019proximal, corbineau2018pipa, corbineau2018text}.
\\\\
The idea behind IPMs is to replace the initial constrained optimization problem by a sequence of unconstrained subproblems of the form: 

\begin{equation}
\min_{x\in\mathbb{R}^n}  ~ f(Hx,y)+\lambda\mathcal{R}(x)+\mu\mathcal{B}(x)
\label{pb:bar_pb}
\end{equation}
where $\mathcal{B}:\mathbb{R}^n \to \mathbb{R} \cup \{+\infty\}$ is the logarithmic barrier function with unbounded derivative at the boundary of the feasible domain:
\begin{equation}
(\forall x\in\mathbb{R}^n)~~~
\mathcal{B}(x) = \left\{  
\begin{array}{r@{\quad}l} 
\displaystyle-\sum\limits_{i=1}^p \ln(c_i(x)) & {\rm{if}}~x\in\mathrm{int}\mathcal{C} \\  
+\infty &  \rm{otherwise,}  
\end{array}\right.
\label{eq:barrier}
\end{equation}
and $\mu \in \ ]0,+\infty[$ is the so--called barrier parameter which vanishes along the minimization process. We assumed that either $f(H\cdot,y)+\lambda\mathcal{R}$ is coercive, or $\mathcal{C}$ is bounded, hence, the set of solutions to (\ref{pb:var_pb2}) is bounded. Since $\rm{int}\mathcal{C}$ is not empty we can apply \cite[Theorem~5(ii)]{wright1992interior} and the existence of solutions to (\ref{pb:bar_pb}) is guaranteed.

\subsection{Proposed iterative schemes}

Thanks to the proximity operator, the IPM from \cite{kaplan1998proximal} does not require any matrix inversion. When the proximity operator is computed in an exact manner, the proposed IPM can be rewritten as Algorithm~\ref{algo:kaplan}, whose convergence has been proven under some assumptions \cite[Theorem 4.1]{kaplan1998proximal}. 

\vspace*{0.4cm}
\begin{algorithm}[H]
 \caption{Exact version of the proximal IPM in \cite{kaplan1998proximal} applied to problem~(\ref{pb:var_pb2}).}
 \label{algo:kaplan}
\begin{algorithmic}
\STATE Let $x_0\in\mathrm{int}\mathcal{C}$, $\underline{\gamma}>0$ and $\left(\gamma_k\right)_{k\in\mathbb{N}}$ be a sequence such that $(\forall k\in\mathbb{N})$ $\underline{\gamma}\leq\gamma_k$;\\
 \FOR{$k=0,1,\ldots$}
 \STATE $x_{k+1}=\mathrm{prox}_{\gamma_k\left(h(\cdot,y,\lambda)+\mu_k\mathcal{B}\right)}\left(x_k\right)$
 \ENDFOR
\end{algorithmic}
\end{algorithm}
\vspace*{0.4cm}

\noindent Algorithm~\ref{algo:kaplan} requires evaluating the proximity operator of the sum of the barrier and the regularized cost function, which can be an issue since, in most of the cases, this operator does not have a closed-form expression. This is the reason why we propose to modify it by introducing a forward step, which leads to Algorithm~\ref{algo:FB_IPM}.

\vspace*{0.4cm}
\begin{algorithm}[H]
 \caption{Proposed forward--backward proximal IPM.}
 \label{algo:FB_IPM}
\begin{algorithmic}
\STATE Let $x_0\in\mathrm{int}\mathcal{C}$, $\underline{\gamma}>0$ and $\left(\gamma_k\right)_{k\in\mathbb{N}}$ be a sequence such that $(\forall k\in\mathbb{N})$ $\underline{\gamma}\leq\gamma_k$;\\
\FOR{$k=0,1,\ldots$}
\STATE $x_{k+1}=\mathrm{prox}_{\gamma_k\mu_k \mathcal{B}}\left(x_k-\gamma_k\nabla_1h\left(x_k,y,\lambda\right)\right) $
\ENDFOR
\end{algorithmic}
\end{algorithm}
\vspace*{0.4cm}

\noindent To the best of our knowledge, there is no available convergence study for Algorithm~\ref{algo:FB_IPM} among the literature of interior-point methods. There exist links between the above algorithm and the diagonal or penalization method introduced in \cite{czarnecki2016splitting}. Indeed, taking $A\equiv 0$ and $\Psi_1\equiv 0$ in \cite{czarnecki2016splitting} leads to Algorithm~\ref{algo:FB_IPM}, whose convergence is proven. However, there are some key differences between both approaches, namely \textit{i)} in \cite{czarnecki2016splitting}, the barrier parameter tends to infinity while it goes to zero in our case, and \textit{ii)} the algorithm in \cite{czarnecki2016splitting} solves a hierarchical minimization problem instead of the constrained optimization problem \eqref{pb:var_pb2}.
It is worth noting that Algorithm~\ref{algo:FB_IPM} only requires computing the proximity operator of the logarithmic barrier. We will provide its expression in Section~\ref{sec:prox_computation} for three different types of constraints.

\subsection{Limitations}
In IPMs, the barrier parameter and stepsize sequences, $\left(\mu_k\right)_{k\in\mathbb{N}}$ and $\left(\gamma_k\right)_{k\in\mathbb{N}}$, are usually set by following some heuristic rules, which ensure the convergence of the method to a minimizer of the considered objective function. However, handcrafted variational formulations do not necessarily capture perceptual image quality well. These heuristics can thus lead to a loss in terms of efficiency and versatility of the resulting restoration schemes.
Moreover, as already mentioned, an accurate setting of the regularization weights is particularly critical in order to obtain a satisfactory image quality when using such penalized restoration approaches. Existing approaches for selecting $\lambda$, which are based on statistical considerations, are usually associated with a substantial increase of the computational cost. 
\\\\\noindent To overcome these limitations, we propose to unfold Algorithm~\ref{algo:FB_IPM} over a given number of iterations and to learn the stepsize, the barrier and the regularization parameters for every iteration in a supervised fashion. Our machine learning method will make use of gradient backpropagation for its training step. The latter requires the derivatives of the proximity operator in Algorithm~\ref{algo:FB_IPM} with respect to its input and to the aforementioned parameters which are to be learned. Therefore, we first conduct an analysis of the proximity operator of the barrier and of its derivatives, for three examples of interest in Section~\ref{sec:prox_computation}.

\section{Proximity operator of the barrier}
\label{sec:prox_computation}
Let $\mathcal{B}$ be defined as in (\ref{eq:barrier}) and for all $\mu>0$, $\gamma>0$ and $x\in\mathbb{R}^n$, let $\varphi$ be defined as follows:
\begin{equation}
\varphi(x,\mu,\gamma)=\mathrm{prox}_{\gamma\mu\mathcal{B}}(x).
\end{equation}
We provide in this section expressions of $\varphi$ and of its derivatives with respect to its input variable $x$ and the involved barrier and stepsize parameters $(\mu,\gamma)$, for three common types of constraints. The latter will be necessary for training the proposed neural network using a gradient backpropagation scheme. 

\subsection{Affine constraints}

Let us first consider the following half-space constraint:
\begin{equation}
\mathcal{C}=\{x\in\mathbb{R}^n~|~a^\top x\leq b\}, \label{eq:affine}
\end{equation}
with $a\in\mathbb{R}^n\setminus\{0\}$ and $b\in\mathbb{R}$. %The proximity operator of the corresponding barrier function is expressed below and we also provide the expressions of the gradients of this operator with respect to its input variable and the involved stepsize parameters.

\begin{proposition}
Let $\gamma>0$, $\mu>0$, and let $\mathcal{B}$ be the function associated to \eqref{eq:affine}, defined as
\begin{equation}
(\forall u\in\mathbb{R}^n)~~~
\mathcal{B}(u) = \left\{  
\begin{array}{r@{\quad}l} 
\displaystyle-\ln(b-a^\top u) & {\rm{if}}~a^\top u < b, \\  
+\infty &  \rm{otherwise.}  
\end{array}\right.
\end{equation}
Then, for every $x\in\mathbb{R}^n$, the proximity operator of $\gamma\mu\mathcal{B}$ at $x$
is given by
\begin{equation}
\varphi(x,\mu,\gamma) =  x+\frac{b-a^\top x -\sqrt{(b-a^\top x)^2+4\gamma\mu\|a\|^2}}{2\|a\|^2}a.
\label{eq:af5}
\end{equation}
In addition,  the Jacobian matrix of $\varphi$ with respect to $x$ and the gradients of $\varphi$ with respect to $\mu$ and $\gamma$ are given by
\begin{equation}
J_\varphi^{(x)}(x,\mu,\gamma)=\mathbb{I}_n-\frac{1}{2\|a\|^2}\left(1+\frac{a^\top x-b}{\sqrt{(b-a^\top x)^2+4\gamma\mu\|a\|^2}}\right)aa^\top,
\label{eq:af5dx}
\end{equation}
\begin{equation}
\nabla_{\varphi}^{(\mu)}(x,\mu,\gamma)=\frac{-\gamma}{\sqrt{(b-a^\top x)^2+4\gamma\mu\|a\|^2}}a,
\label{eq:af5dmu}
\end{equation}
and
\begin{equation}
\nabla_{\varphi}^{(\gamma)}(x,\mu,\gamma)=\frac{-\mu}{\sqrt{(b-a^\top x)^2+4\gamma\mu\|a\|^2}}a,
\label{eq:af5dga}
\end{equation}
where $\mathbb{I}_n\in\mathbb{R}^{n\times n}$ denotes the identity matrix.
\label{prop:affine}
\end{proposition}

\begin{proof}
The expression for the proximity operator (\ref{eq:af5}) directly follows from \cite[Example~24.40]{bauschke2017convex}, \cite[Proposition~24.8~(v)]{bauschke2017convex} and \cite[Corollary~24.15]{bauschke2017convex}. Taking the derivative of (\ref{eq:af5}) with respect to $x$, $\mu$ and $\gamma$ leads to (\ref{eq:af5dx})--(\ref{eq:af5dga}).
\end{proof}

\subsection{Hyperslab constraints}
\label{sec:hyperslab}

We now consider the following hyperslab set:
\begin{equation}
\mathcal{C}=\{x\in\mathbb{R}^n~|~b_{\rm m}\leq a^\top x\leq b_{\rm M}\}, \label{eq:hyper}
\end{equation}
where $a\in\mathbb{R}^n\setminus\{0\}$, $b_{\rm m}\in\mathbb{R}$ and $b_{\rm M}\in\mathbb{R}$ with $b_{\rm m} < b_{\rm M}$. %We provide below the proximity operator of the corresponding barrier function, along with the expressions of the gradients of this operator with respect to its input variable and the involved stepsize parameters. 

\begin{proposition}
Let $\gamma>0$, $\mu>0$, and let $\mathcal{B}$ be the barrier function associated to \eqref{eq:hyper}, defined as
\begin{equation}
(\forall u\in\mathbb{R}^n)~~~
\mathcal{B}(u) = \left\{  
\begin{array}{r@{\quad}l} 
\displaystyle-\ln(b_{\rm M}-a^\top u)-\ln(a^\top u-b_{\rm m})& {\rm{if}}~b_{\rm m}< a^\top u< b_{\rm M}, \\  
+\infty &  \rm{otherwise.}  
\end{array}\right.
\end{equation}
Then, for every $x\in\mathbb{R}^n$, the proximity operator of $\gamma\mu\mathcal{B}$ at $x$ is given by
\begin{equation}
\varphi(x,\mu,\gamma)=x+\frac{\kappa(x,\mu,\gamma)-a^\top x}{\|a\|^2}a,
\label{eq:hyprop1}
\end{equation}
where $\kappa(x,\mu,\gamma)$ is the unique solution in $]b_{\rm m},b_{\rm M}[$, of the following cubic equation:
\begin{equation}
\begin{split}
0=&\ z^3 -(b_{\rm m}+b_{\rm M}+a^\top x)z^2 +(b_{\rm m}b_{\rm M}+a^\top x(b_{\rm m}+b_{\rm M})-2\gamma\mu\|a\|^2)z
\\&-b_{\rm m}b_{\rm M}a^\top x+\gamma\mu(b_{\rm m}+b_{\rm M})\|a\|^2.
\end{split}
\label{eq:hyprop2}
\end{equation}
In addition,  the Jacobian matrix of $\varphi$ with respect to $x$ and the gradients of $\varphi$ with respect to $\mu$ and $\gamma$ are given by
\begin{equation}
J_\varphi^{(x)}(x,\mu,\gamma)=\mathbb{I}_n+\frac{1}{\|a\|^2}\left(\dfrac{(b_{\rm M}-\kappa(x,\mu,\gamma))(b_{\rm m}-\kappa(x,\mu,\gamma))}{\eta(x,\mu,\gamma)}-1\right)a a^\top,
\label{eq:hyprop3}
\end{equation}
\begin{equation}
\nabla_{\varphi}^{(\mu)}(x,\mu,\gamma)=\frac{-\gamma(b_{\rm m}+b_{\rm M}-2\kappa(x,\mu,\gamma))}{\eta(x,\mu,\gamma)}a,
\end{equation}
and
\begin{equation}
\nabla_{\varphi}^{(\gamma)}(x,\mu,\gamma)=\frac{-\mu(b_{\rm m}+b_{\rm M}-2\kappa(x,\mu,\gamma))}{\eta(x,\mu,\gamma)}a,
\label{eq:hyprop5}
\end{equation}
where
\begin{equation}
\begin{split}
\eta(x,\mu,\gamma)=&\ (b_{\rm M}-\kappa(x,\mu,\gamma))(b_{\rm m}-\kappa(x,\mu,\gamma))
\\&-(b_{\rm m}+b_{\rm M}-2\kappa(x,\mu,\gamma))(\kappa(x,\mu,\gamma)-a^\top x)-2\gamma\mu\|a\|^2.
\end{split}
\label{eq:hyprop6}
\end{equation}
\label{prop:hyperslab}
\end{proposition}

\begin{proof}
Let $x\in\mathbb{R}^n$, $\gamma>0$, and $\mu>0$. The expression for the proximity operator (\ref{eq:hyprop1}) follows from \cite[Example~4.15]{chaux2007} and \cite[Corollary~24.15]{bauschke2017convex}. Let $F$ be defined as follows:
\begin{equation}
F(x,\mu,\gamma,z)=(b_{\rm M}-z)(b_{\rm m}-z)(z-a^\top x)+\gamma\mu(b_{\rm M}+b_{\rm m}-2z)\|a\|^2,
\label{eq:hyp_F}
\end{equation}
for $z \in \ ]b_{\rm m},b_{\rm M}[$. Expanding (\ref{eq:hyp_F}) gives the following:
\begin{equation}
\begin{split}
F(x,\mu,\gamma,z)=&\ z^3-(a^\top x+b_{\rm m}+b_{\rm M})z^2+(b_{\rm m}b_{\rm M}+a^\top x(b_{\rm m}+b_{\rm M})-2\gamma\mu\|a\|^2)z\\&-b_{\rm m}b_{\rm M}a^\top x+\gamma\mu(b_{\rm m}+b_{\rm M})\|a\|^2.
\end{split}
\end{equation}
Hence, by definition of $\kappa(x,\mu,\gamma)$, we have $F(x,\mu,\gamma,\kappa(x,\mu,\gamma))=0$. In addition, the derivative of $F$ with respect to its last variable is equal to
\begin{equation}
\nabla F^{(z)}(x,\mu,\gamma,z)=(b_{\rm M}-z)(b_{\rm m}-z)-(b_{\rm m}+b_{\rm M}-2z)(z-a^\top x)-2\gamma\mu\|a\|^2.
\end{equation}
By construction, $(b_{\rm M}-\kappa(x,\mu,\gamma))(b_{\rm m}-\kappa(x,\mu,\gamma))<0$. Moreover, $-2\gamma\mu\|a\|^2<0$ and, since $F(x,\mu,\gamma,\kappa(x,\mu,\gamma))=0$, it follows that $(b_{\rm m}+b_{\rm M}-2\kappa(x,\mu,\gamma))$ and $\kappa(x,\mu,\gamma)-a^\top x$ share the same sign. Hence, 
\begin{equation}
\eta(x,\mu,\gamma) = \nabla F^{(z)}(x,\mu,\gamma,\kappa(x,\mu,\gamma))\neq 0.
\end{equation}
From the implicit function theorem \cite[Theorem~1B.1]{dontchev2009implicit}, we deduce that the gradient of $\kappa$ with respect to $x$ and the partial derivatives of $\kappa$ with respect to $\mu$ and $\gamma$ exist and are equal to
\begin{equation}
\nabla \kappa^{(x)}(x,\mu,\gamma)=\dfrac{(b_{\rm M}-\kappa(x,\mu,\gamma))(b_{\rm m}-\kappa(x,\mu,\gamma))}{\eta(x,\mu,\gamma)}a,
\label{eq:hypproof1}
\end{equation}
\begin{equation}
\nabla \kappa^{(\mu)}(x,\mu,\gamma)=\frac{-\gamma\|a\|^2(b_{\rm m}+b_{\rm M}-2\kappa(x,\mu,\gamma))}{\eta(x,\mu,\gamma)},
\end{equation}
and
\begin{equation}
\nabla \kappa^{(\gamma)}(x,\mu,\gamma)=\frac{-\mu\|a\|^2(b_{\rm m}+b_{\rm M}-2\kappa(x,\mu,\gamma))}{\eta(x,\mu,\gamma)}.
\label{eq:hypproof3}
\end{equation}
Differentiating (\ref{eq:hyprop1}) with respect to $x$, $\mu$ and $\gamma$ and using (\ref{eq:hypproof1})--(\ref{eq:hypproof3}) yields (\ref{eq:hyprop3})--(\ref{eq:hyprop5}).
\end{proof}

\noindent Note that the three roots of (\ref{eq:hyprop2}) can easily be computed using the Cardano formula. The graph of the resulting proximity operator is plotted on Figure~\ref{fig:prox_barriers} (left) for $n=1$, $a=1$, $b_{\rm m}=0$, $b_{\rm M}=1$, and various values for $\gamma\mu$.

\begin{figure}
\setlength\tabcolsep{4pt}
\centering
\begin{tabularx}{\textwidth}{cc}
\includegraphics[width=0.45\textwidth]{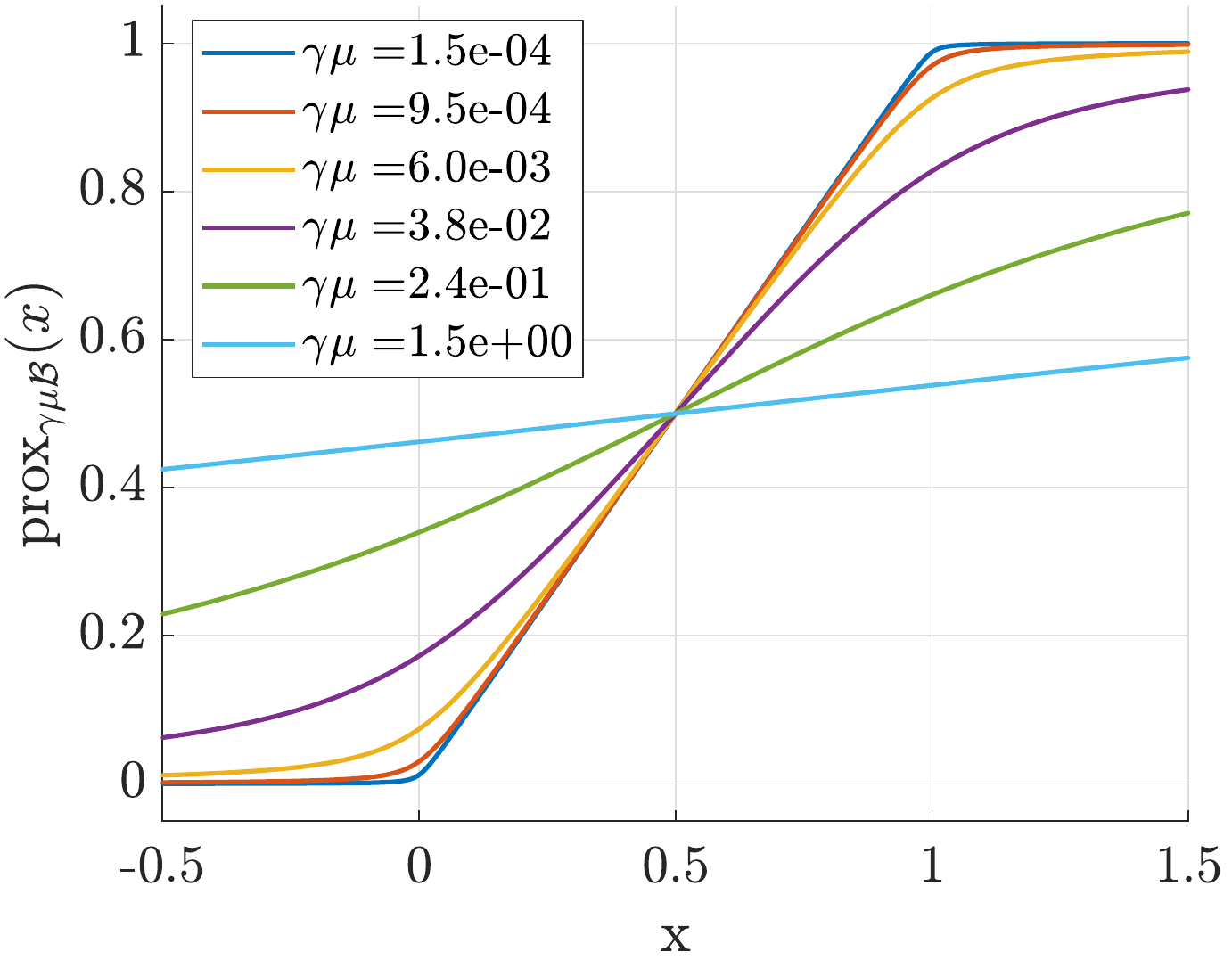}&
\includegraphics[width=0.49\textwidth]{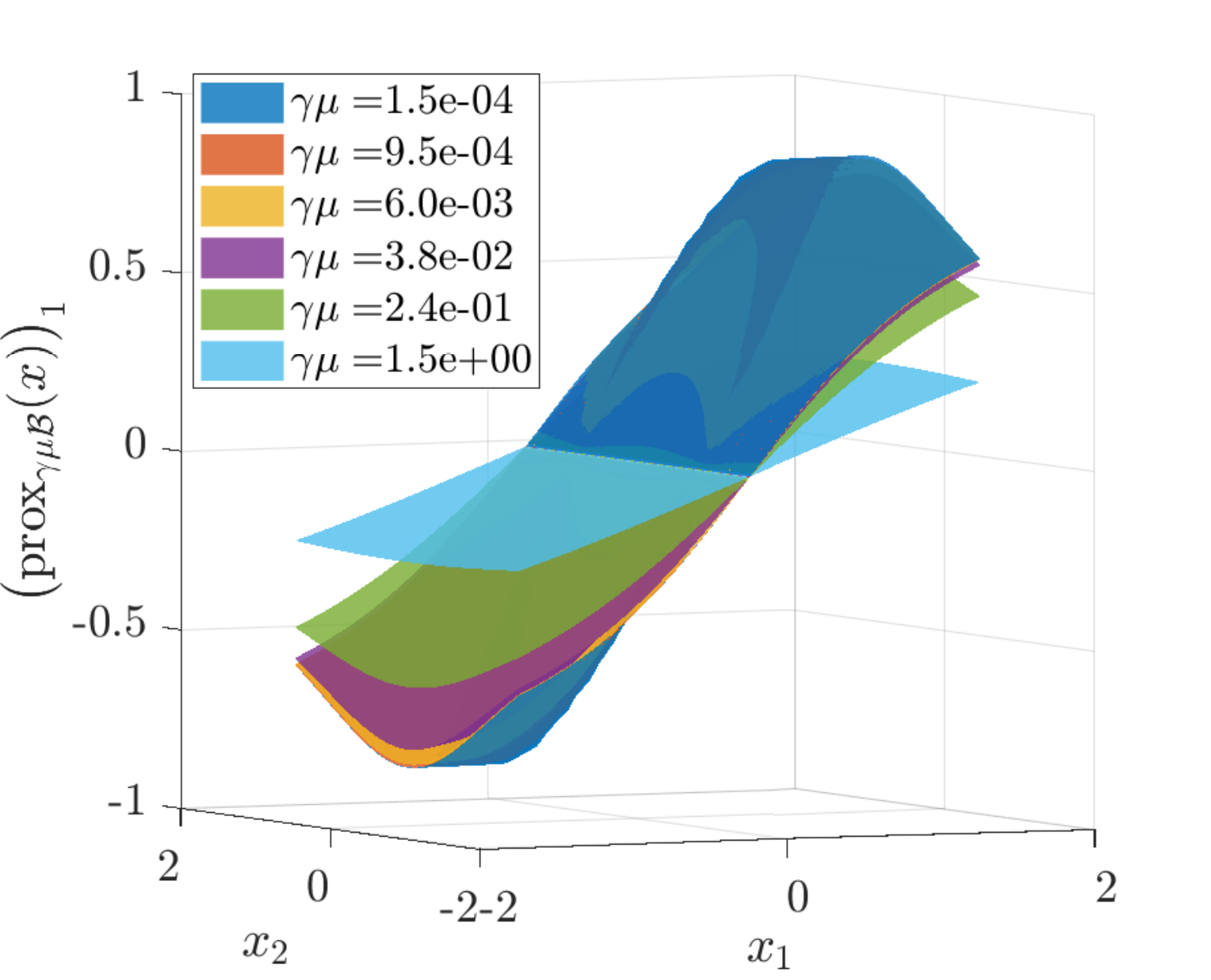}
\end{tabularx}
\caption{Proximity operator of the logarithmic barrier: $\mathrm{prox}_{\gamma\mu\mathcal{B}}(x)$ for hyperslab constraint as in Section~\ref{sec:hyperslab} with $b_{\rm m}=0$ and $b_{\rm M}=1$ ({\em left}), $\left(\mathrm{prox}_{\gamma\mu\mathcal{B}}(x)\right)_1$ for a constraint on the $\ell_2$-norm as in Section~\ref{sec:ll2} with $\alpha=0.7$ ({\em right}).}
\label{fig:prox_barriers}
\end{figure}

\subsection{Bounded $\ell_2$-norm}
\label{sec:ll2}

We now consider the case when the feasible set in (\ref{pb:var_pb2}) is {\color{blue}a} Euclidean ball
\begin{equation}
\mathcal{C}=\{x\in\mathbb{R}^n~|~\|x-c\|^2\leq \alpha\}, \label{eq:eucli}
\end{equation}
with $\alpha>0$ and $c\in\mathbb{R}^n$. %We provide below the proximity operator of the corresponding barrier function, along with the expressions of the gradients of this operator with respect to its input variable and the involved stepsize parameters.

\begin{proposition}
Let $\gamma>0$ and let $\mu>0$. Let $\mathcal{B}$ be the barrier function associated to \eqref{eq:eucli}, defined as
\begin{equation}
(\forall u\in\mathbb{R}^n)~~~
\mathcal{B}(u) = \left\{  
\begin{array}{r@{\quad}l} 
\displaystyle-\ln(\alpha-\|u-c\|^2)& {\rm{if}}~\|u-c\|^2<\alpha, \\  
+\infty &  \rm{otherwise.}  
\end{array}\right.
\label{eq:barrier_ll2}
\end{equation}
Then, for every $x\in\mathbb{R}^n$, the proximity operator of $\gamma\mu\mathcal{B}$ at $x$ is given by
\begin{equation}
\varphi(x,\mu,\gamma)=c+\frac{\alpha-\kappa(x,\mu,\gamma)^2}{\alpha-\kappa(x,\mu,\gamma)^2+2\gamma\mu}(x-c),\label{eq:bl1}
\end{equation}
where $\kappa(x,\mu,\gamma)$ is the unique solution in $[0,\sqrt{\alpha}[$ of the cubic equation:
\begin{equation}
0=z^3-\|x-c\|z^2-(\alpha+2\gamma\mu)z+\alpha\|x-c\|.
\label{eq:bl2}
\end{equation}
In addition, the Jacobian matrix of $\varphi$ with respect to $x$ and the gradients of $\varphi$ with respect to $\mu$ and $\gamma$ are given by
\begin{equation}
J_\varphi^{(x)}(x,\mu,\gamma)=\frac{\alpha-\|\varphi(x,\mu,\gamma)-c\|^2}{\alpha-\|\varphi(x,\mu,\gamma)-c\|^2+2\gamma\mu}M(x,\mu,\gamma),
\label{eq:bl3}
\end{equation}
\begin{equation}
\nabla_{\varphi}^{(\mu)}(x,\mu,\gamma)=\frac{-2\gamma}{\alpha-\|\varphi(x,\mu,\gamma)-c\|^2+2\gamma\mu}M(x,\mu,\gamma)(\varphi(x,\mu,\gamma)-c),
\label{eq:bl4}
\end{equation}
and
\begin{equation}
\nabla_{\varphi}^{(\gamma)}(x,\mu,\gamma)=\frac{-2\mu}{\alpha-\|\varphi(x,\mu,\gamma)-c\|^2+2\gamma\mu}M(x,\mu,\gamma)(\varphi(x,\mu,\gamma)-c),
\label{eq:bl5}
\end{equation}
where
\begin{equation}
M(x,\mu,\gamma)=\mathbb{I}_n-\frac{2(x-\varphi(x,\mu,\gamma))(\varphi(x,\mu,\gamma)-c)^\top}{\alpha-3\|\varphi(x,\mu,\gamma)-c\|^2+2\gamma\mu+2(\varphi(x,\mu,\gamma)-c)^\top (x-c)}.
\end{equation}
\label{prop:l2}
\end{proposition}

\begin{proof}
Let $x\in\mathbb{R}^n$, $\gamma>0$, $\mu>0$. 
%For the sake of concision, we will make use of the shorter notation $v=\varphi(x,\mu,\gamma)$. 
Let us first consider the case when $c=0$. We denote with $\varphi_0$ the following proximity operator:
\begin{equation}
\varphi_0(x,\mu,\gamma)=\underset{u\in\mathrm{int}\mathcal{C}}{\mathrm{argmin}} \ \frac{1}{2}\|x-u\|^2-\gamma\mu\ln(\alpha-\|u\|^2).
\label{eq:bl6}
\end{equation}
Hence, $\|\varphi_0(x,\mu,\gamma)\|^2<\alpha$ and $\varphi_0(x,\mu,\gamma)$ is a solution to the following equation: 
\begin{equation}
0=\varphi_0(x,\mu,\gamma)-x+\frac{2\gamma\mu}{\alpha-\|\varphi_0(x,\mu,\gamma)\|^2}\varphi_0(x,\mu,\gamma).
\label{eq:bl7}
\end{equation}
Since $\alpha-\|\varphi_0(x,\mu,\gamma)\|^2+2\gamma\mu>0$, (\ref{eq:bl7}) becomes
\begin{equation}
\varphi_0(x,\mu,\gamma)=\frac{\alpha-\|\varphi_0(x,\mu,\gamma)\|^2}{\alpha-\|\varphi_0(x,\mu,\gamma)\|^2+2\gamma\mu}x.
\label{eq:bl8}
\end{equation}
By taking the norm in both sides of (\ref{eq:bl8}), we deduce that $\|\varphi_0(x,\mu,\gamma)\|=\kappa(x,\mu,\gamma)$ is a solution to the cubic equation (\ref{eq:bl2}). Since the proximity operator at a given $x$ is uniquely defined, there exists only one real solution to (\ref{eq:bl2}) which belongs to $[0,\sqrt{\alpha}[$. Plugging the latter into (\ref{eq:bl8}) leads to (\ref{eq:bl1}). The analysis when $c\neq 0$ is deduced from the case $c=0$ by using \cite[Proposition~24.8~(v)]{bauschke2017convex}: the proximity operator of $\gamma\mu\mathcal{B}$ at $x$ is given by 
\begin{equation}
\varphi(x,\mu,\gamma)=c+\varphi_0(x-c,\mu,\gamma).
\label{eq:link_phi_phi0}
\end{equation}
%: let $\varphi_{c=0}$ be the proximity operator when $c=0$, then $\varphi(x,\mu,\gamma)=c+\varphi_{c=0}(x-c,\mu,\gamma)$. 
\\\noindent Let us study the derivatives of $\varphi_0$. For every $v\in\mathbb{R}^n$, let $F$ be defined as
\begin{equation}
F(x,\mu,\gamma,v)=(\alpha-\|v\|^2)(v-x)+2\gamma\mu v.
\end{equation} 
The Jacobian of $F$ with respect to its last variable is equal to
\begin{equation}
J_F^{(v)}(x,\mu,\gamma,v)=(\alpha-\|v\|^2+2\gamma\mu)\mathbb{I}_n+2(x-v)v^{\top}.
\end{equation}
Since $\alpha-\|\varphi_0(x,\mu,\gamma)\|^2>0$, according to the Sherman--Morrison Lemma \cite{bartlett1951inverse}, $J_F^{(v)}(x,\mu,\gamma,\varphi_0(x,\mu,\gamma))$ is invertible if and only if
\begin{equation}
\alpha-\|\varphi_0(x,\mu,\gamma)\|^2+2\gamma\mu +2\varphi_0(x,\mu,\gamma)^\top(x-\varphi_0(x,\mu,\gamma))\neq 0.
\end{equation}
Furthermore, it follows from (\ref{eq:bl7}) that
\begin{equation}
F(x,\mu,\gamma,\varphi_0(x,\mu,\gamma))=0.
\label{eq:bl9}
\end{equation}
Applying $\varphi_0(x,\mu,\gamma)^\top$ on (\ref{eq:bl9}) leads to $\varphi_0(x,\mu,\gamma)^\top(x-\varphi_0(x,\mu,\gamma))\geq 0$. In addition, $\alpha-\|\varphi_0(x,\mu,\gamma)\|^2+2\gamma\mu>0$. Hence, $J_F^{(v)}(x,\mu,\gamma,\varphi_0(x,\mu,\gamma))$ is invertible and its inverse is given by the Sherman--Morrison formula:
\begin{equation}
\begin{split}
J_F^{(v)}(x,\mu,\gamma,\varphi_0(x,\mu,\gamma))^{-1}=&\frac{1}{\alpha-\|\varphi_0(x,\mu,\gamma)\|^2+2\gamma\mu}\times\\
&\left[\mathbb{I}_n-\frac{2(x-\varphi_0(x,\mu,\gamma))\varphi_0(x,\mu,\gamma)^\top}{\alpha-3\|\varphi_0(x,\mu,\gamma)\|^2+2\gamma\mu+2\varphi_0(x,\mu,\gamma)^\top x}\right].
\end{split}
\end{equation}
From the implicit function theorem \cite[Theorem~1B.1]{dontchev2009implicit} we deduce that the Jacobian of $\varphi_0$ with respect to $x$  and the gradients of $\varphi_0$ with respect to $\mu$ and $\gamma$ exist and are equal to
\begin{equation}
J_{\varphi_0}^{(x)}(x,\mu,\gamma)=-J_F^{(v)}(x,\mu,\gamma,\varphi_0(x,\mu,\gamma))^{-1}J_F^{(x)}(x,\mu,\gamma,\varphi_0(x,\mu,\gamma)),
\end{equation}
\begin{equation}
\nabla_{\varphi_0}^{(\mu)}(x,\mu,\gamma)=-J_F^{(v)}(x,\mu,\gamma,\varphi_0(x,\mu,\gamma))^{-1}\nabla_F^{(\mu)}(x,\mu,\gamma,\varphi_0(x,\mu,\gamma)),
\end{equation}
and
\begin{equation}
\nabla_{\varphi_0}^{(\gamma)}(x,\mu,\gamma)=-J_F^{(v)}(x,\mu,\gamma,\varphi_0(x,\mu,\gamma))^{-1}\nabla_F^{(\gamma)}(x,\mu,\gamma,\varphi_0(x,\mu,\gamma)).
\end{equation}
When $c\neq 0$, the derivatives of $\varphi$ are deduced from those of $\varphi_0$ using (\ref{eq:link_phi_phi0}):
\begin{equation}
J_{\varphi}^{(x)}(x,\mu,\gamma)=-J_F^{(v)}(x-c,\mu,\gamma,\varphi(x,\mu,\gamma)-c)^{-1}J_F^{(x)}(x-c,\mu,\gamma,\varphi(x,\mu,\gamma)-c),
\end{equation}
\begin{equation}
\nabla_{\varphi}^{(\mu)}(x,\mu,\gamma)=-J_F^{(v)}(x-c,\mu,\gamma,\varphi(x,\mu,\gamma)-c)^{-1}\nabla_F^{(\mu)}(x-c,\mu,\gamma,\varphi(x,\mu,\gamma)-c),
\end{equation}
and
\begin{equation}
\nabla_{\varphi}^{(\gamma)}(x,\mu,\gamma)=-J_F^{(v)}(x-c,\mu,\gamma,\varphi(x,\mu,\gamma)-c)^{-1}\nabla_F^{(\gamma)}(x-c,\mu,\gamma,\varphi(x,\mu,\gamma)-c),
\end{equation}
which lead to (\ref{eq:bl3})-(\ref{eq:bl5}).
\end{proof}

\noindent Similarly to the previous case, the three solutions to (\ref{eq:bl2}) can be obtained by using the Cardano formula. The form of the resulting proximity operator for $n=2$ is plotted on Figure~\ref{fig:prox_barriers} (right) for $\alpha=0.7$, $c = 0$, and several values of $\gamma\mu$ and $x$; for symmetry reasons, only the first component $\left(\mathrm{prox}_{\gamma\mu\mathcal{B}}(x)\right)_1$ is represented.
\\\\
As shown in this section, the proximity operator of the barrier is easily computable and differentiable for several classic types of constraints. Next, we detail the proposed approach in Section~\ref{sec:net}.

\section{iRestNet architecture}
\label{sec:net}
\subsection{Overview}

Our proposal is to adopt a supervised learning strategy in order to determine, from a training set of images, an optimal setting for the parameters of Algorithm~\ref{algo:FB_IPM}, which should lead to an optimal image restoration quality. To this aim, Algorithm~\ref{algo:FB_IPM} is unfolded over $K$ iterations and the regularization parameter $\lambda$ is untied across the network, so as to provide more flexibility to the approach \cite{hershey2014deep}. The update rule at a given iteration $k\in\{0,\ldots,K-1\}$ reads
\begin{equation}
x_{k+1} = \mathcal{A}\left(x_k,\mu_k,\gamma_k,\lambda_k\right) 
\end{equation}
with
\begin{equation}
\mathcal{A}\left(x_k,\mu_k,\gamma_k,\lambda_k\right) =\mathrm{prox}_{\gamma_k \mu_k \mathcal{B}}\left(x_k-\gamma_k\nabla_1h\left(x_k,y,\lambda_k\right)\right). 
\label{eq:mathcalA}
\end{equation}

\begin{figure}
\includegraphics[width=1\textwidth]{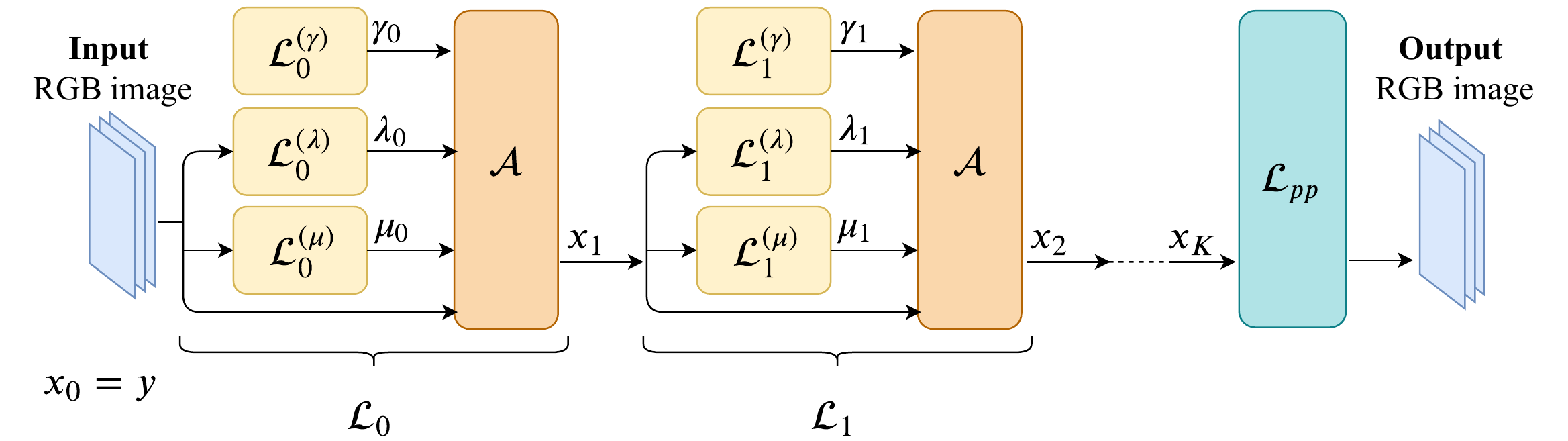}
\caption{iRestNet global architecture.}
\label{fig:net_archi}
\end{figure}

\noindent For every $k\in\{0,\ldots,K-1\}$, we build the $k$-th \textit{layer} $\mathcal{L}_k$ as the association of three hidden structures, $\mathcal{L}^{(\mu)}_k$, $\mathcal{L}^{(\gamma)}_k$ and $\mathcal{L}^{(\lambda)}_k$, followed by the update $\mathcal{A}$. Structures $\mathcal{L}^{(\mu)}_k$, $\mathcal{L}^{(\gamma)}_k$, and $\mathcal{L}^{(\lambda)}_k$ aim at inferring the barrier parameter $\mu_k$, the stepsize $\gamma_k$ and the regularization weight $\lambda_k$, respectively. Since a finite number $K$ of layers (i.e., updates) is used, the convergence of the resulting scheme is not an issue. Note that we also allow in our framework the use of a post--processing step after going through the $K$ layers, that will be denoted as $\mathcal{L}_{\rm pp}$. The resulting architecture is depicted in Figure~\ref{fig:net_archi}.

\subsection{Hidden structures}
Let us now provide more details about the hidden structures. For every $k\in\{0,\ldots,K-1\}$, the outputs $(\mu_k,\gamma_k,\lambda_k)$ of the structures $\mathcal{L}^{(\mu)}_k$, $\mathcal{L}^{(\gamma)}_k$, and $\mathcal{L}^{(\lambda)}_k$ must be positive. To enforce such constraint, we use the Softplus function \cite{dugas2001incorporating}, defined below, which can be viewed as a smooth approximation of the ReLU activation function:
\begin{equation}
(\forall z\in\mathbb{R})\quad\mathrm{Softplus}(z)=\ln(1+\exp(z)).
\end{equation}
Unlike the ReLU, the gradient of Softplus is never strictly equal to zero, which, given our architecture, helps to propagate the gradient through the network. 
The stepsize is estimated as follows, 
\begin{equation}
\gamma_k=\mathcal{L}_k^{(\gamma)}=\mathrm{Softplus}\left(a_k\right),
\label{eq:L_k_gamma}
\end{equation}
where $a_k$ is a scalar parameter of the network learned during training. 
The barrier parameter is obtained using two convolutional and average pooling layers followed by a fully connected layer. The detailed architecture of $\mathcal{L}^{(\mu)}_k$ is depicted in Figure~\ref{fig:lkmu}. 
\\
Traditional methods for estimating the regularization parameter generally depend on the signal-to-noise ratio and on the image statistics \cite{vogel2002computational}. For most applications the noise level is unknown and can be estimated, for instance, by applying a median filter over the wavelet diagonal coefficients of the image \cite{ramadhan2017image}. This strategy is used in the numerical experiments presented in Section~\ref{sec:exp}. The advantage is to yield a network which can handle datasets for which the signal-to-noise ratio is unknown and can vary within a reasonable range.
%
%an alternate approach with the aim to handle datasets for which the signal-to-noise ratio is allowed to vary within a reasonable range. For every image, the regularization parameter $\lambda_k$ is only inferred from the image statistics and we let the network "learn" the noise level, by introducing a multiplicative factor, for instance. 
% and we want iRestNet not to be noise-specific. Indeed, the proposed architecture can 
The expression of $\mathcal{L}_k^{(\lambda)}$ is then problem--dependent since its expression depends on the regularization function $\mathcal{R}$. A specific example is given in Section~\ref{sec:exp} for the total variation regularization function. 
\\
Regarding the post-processing step $\mathcal{L}_{\rm pp}$, its detailed architecture also depends on the task to be performed. An example is provided in Section~\ref{sec:exp} for the case of deblurring: the purpose of $\mathcal{L}_{\rm pp}$ is then to remove remaining artifacts using convolutional layers, residual learning, batch normalization, and dilation.

\begin{figure}
\centering
\includegraphics[width=0.8\textwidth]{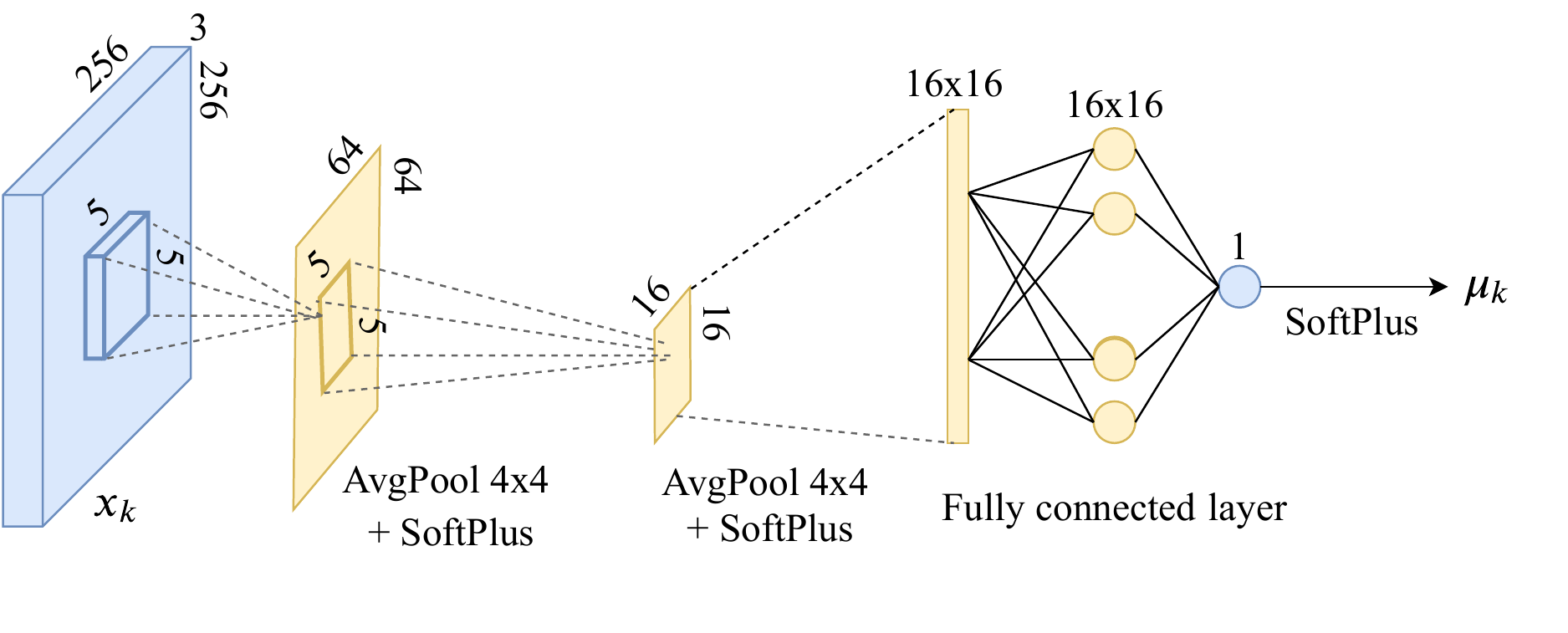}
\caption{Architecture of $\mathcal{L}_k^{(\mu)}$.}
\label{fig:lkmu}
\end{figure}

\subsection{Differential calculus}
To train the neural network presented in Figure~\ref{fig:net_archi} using gradient descent, one needs to compute the gradient of $x_K$ with respect to the different parameters of the network. The chain rule can be applied since most of the steps in the network correspond to operators having straightforward derivatives. However, particular care should be taken when
 differentiating $\mathcal{A}$. Since $f$ and $\mathcal{R}$ are assumed to be twice differentiable, the only area of concern is related to $\mathrm{prox}_{\gamma\mu\mathcal{B}}$. If $\mathrm{prox}_{\gamma\mu\mathcal{B}}$ is simple enough, automatic differentiation~\cite{paszke2017autodiff} can be used. Otherwise, as shown in Section~\ref{sec:prox_computation}, for common examples of barrier functions, the differential of this term is well-defined. The corresponding expressions for the derivatives are provided in Propositions~\ref{prop:affine}--\ref{prop:l2}.

\section{Network stability}
\label{sec:network_stability}
One critical issue concerning neural networks is to guarantee that their performance remains acceptable when the input is perturbed. For example, the authors of \cite{szegedy2013intriguing} show that the class prediction made by AlexNet can be arbitrarily changed by using small nonrandom perturbations on the test image. A recent work \cite{combettes2018deep} provides a theoretical framework which enables to evaluate the robustness of a network. In this section, we will focus on a subclass of problem \eqref{pb:var_pb2} where both $f(\cdot,y)$ and $\mathcal{R}$ are quadratic functions. After highlighting the similarities between the proposed architecture and generic feedforward networks in that case, we will give explicit conditions under which the robustness of the proposed architecture is ensured.

%%%%%%%%%%%%%%%%%%%%%%%%%%%%%%%%%%%%%%%%%%%%%%%%%%%%%%%%%%%%%%
\subsection{Relation to generic deep neural networks}
\label{subsec:relation_gen_net}
Although the proposed architecture may seem specific to Algorithm~\ref{algo:FB_IPM}, it is actually very similar to generic feedforward neural networks. Classical feedforward (acyclic) architectures \cite{schmidhuber2015deep} can be expressed as $R_{K-1}\circ(W_{K-1}\cdot + b_{K-1})\circ\cdots\circ R_0\circ(W_0\cdot + b_0)$, where $(R_k)_{0\leq k\leq K-1}$ are nonlinear activation functions, $(W_k)_{0\leq k\leq K-1}$ are weight operators and $(b_k)_{0\leq k\leq K-1}$ are bias parameters. Let us show that iRestNet actually shares a similar structure. For the sake of simplicity, we will consider the variational problem,
\begin{equation}
\minimize\limits_{x\in \mathcal{C}}\frac{1}{2}\|Hx-y\|^2+\frac{\lambda}{2}\|Dx\|^2,
\label{pb:net_stability}
\end{equation}
where $y\in\mathbb{R}^n$, $H\in\mathbb{R}^{n\times n}$, $D\in\mathbb{R}^{n\times n}$, and $\mathcal{C}$ is defined as in \eqref{def:feasible_set}. Moreover, we assume that no post--processing layer $\mathcal{L}_{\rm pp}$ is used. Following the notation of Section~\ref{sec:net}, $(\forall k\in\{0,\ldots,K-1\})$ $(\mu_k,\gamma_k,\lambda_k)$ are given positive real numbers, $K$ being the number of layers of the network. Then, for every $k\in\{0,\ldots,K-1\}$, layer $\mathcal{L}_k$ corresponds to the following update,
\begin{align}
x_{k+1}&= \mathrm{prox}_{\gamma_k\mu_k\mathcal{B}}\left(x_{k}-\gamma_k\left(H^\top\left(Hx_{k}-y\right)+\lambda_kD^\top Dx_{{k}}\right)\right)\nonumber\\
&=\mathrm{prox}_{\gamma_k\mu_k\mathcal{B}}\left(\left[\mathbb{I}_n-\gamma_k\left(H^\top H +\lambda_k D^\top D\right)\right]x_{k}+\gamma_kH^\top y\right),
\end{align}
where $\mathcal{B}$ is defined as in \eqref{eq:barrier}. For every $k\in\{0,\ldots,K-1\}$, we set 
\begin{equation}
W_{k}=\mathbb{I}_n-\gamma_{k}\left(H^\top H +\lambda_{k} D^\top D\right),~~b_{k} = \gamma_{k}H^\top y,~~\mathrm{and}~~R_{k}=\mathrm{prox}_{\gamma_{k}\mu_{k}\mathcal{B}}.
\label{eq:weights_bias_act}
\end{equation}
Then, the $K$-layer network $\mathcal{L}_{K-1}\circ\cdots\circ\mathcal{L}_0$ is equivalent to $R_{K-1}\circ(W_{K-1}\cdot + b_{K-1})\circ\cdots\circ R_0\circ(W_0\cdot + b_0)$, where $(W_k)_{0\leq k\leq K-1}$ and $(b_k)_{0\leq k\leq K-1}$ are interpreted as weight operators and bias parameters, respectively. The operators $(R_k)_{0\leq k \leq K-1}$ defined in \eqref{eq:weights_bias_act} can be viewed as specific activation functions since, as shown in \cite{combettes2018deep}, every standard activation function can be derived from a proximity operator. In addition, using \cite[Proposition~24.8(iii)]{bauschke2017convex}, for every $k\in\{0,\ldots,K-1\}$, $R_k$ can be re-written as the sum of a \textit{proximal activation operator} \cite[Definition~2.20]{combettes2018deep} and a bias.

%\begin{equation}
%R_k=\mathrm{prox}_{\tilde{\mathcal{B}}_k}+\tilde{b}_k,
%\end{equation}
%where $\tilde{b}_k=\mathrm{prox}_{\gamma_k\mu_k\mathcal{B}}(0)$ and $(\forall x\in\mathbb{R}^n)$ $\tilde{\mathcal{B}}_k(x)=\gamma_k\mu_k\mathcal{B}(x+\tilde{b}_k)+x^\top\tilde{b}_k$. The zero vector is therefore a fixed point of $\mathrm{prox}_{\tilde{\mathcal{B}}_k}$ with $\tilde{\mathcal{B}}_k\in\Gamma_0(\mathbb{R}^n)$. Hence, for every $k\in\{0,\ldots,K-1\}$, $R_k$ is the sum of a \textit{stable activation operator} \cite[Definition~2.20]{combettes2018deep} 
%$\mathrm{prox}_{\tilde{\mathcal{B}}_k}$ and a bias $\tilde{b}_k$. 

%%%%%%%%%%%%%%%%%%%%%%%%%%%%%%%%%%%%%%%%%%%%%%%%%%%%%%%%%%%%
\subsection{Preliminary results}
Before stating our main stability theorem, we recall the result from \cite[Lemma~3.3]{combettes2018deep} in Proposition~\ref{prop:theta} below. We then derive Proposition~\ref{prop:condition31}, which will appear useful when addressing the robustness of the global network. In the following, $\mathcal{S}_n$ denotes the set of symmetric matrices in $\mathbb{R}^{n\times n}$ and, for every $W \in \mathbb{R}^{n\times n}$, $\| W \|$ denotes its spectral norm.

\begin{proposition}{\rm \cite{combettes2018deep}}
Let $K\geq 1$ be an integer and set $\theta_{-1}=1$. 
For every $k\in\{0,\ldots,K-1\}$, let $W_k\in\mathbb{R}^{n\times n}$ and let $\theta_k$ be defined by 
\begin{equation}
\begin{split}
\theta_k=\|W_k\circ\cdots\circ W_0\| + \sum_{\ell=0}^{k-1}&\sum_{0\leq j_0<\cdots<j_\ell\leq k-1}\|W_k\circ\cdots\circ W_{j_\ell+1}\|\times\\
&\|W_{j_\ell}\circ\cdots\circ W_{j_{\ell-1}+1}\|\cdots\|W_{j_0}\circ\cdots\circ W_0\|.
\end{split}
\end{equation}
Then, for every $k\in\{0,\ldots,K-1\}$,
$\theta_{k}=\sum_{\ell=0}^{k}\theta_{\ell-1}\left\|W_k\circ\ldots\circ W_l\right\|$.
\label{prop:theta}
\end{proposition}

\begin{proposition}
Let $K\geq 1$, $\theta>0$, and $\alpha\in[1/2,1]$. 
Let $W\in \mathcal{S}_{n}$ and let $\beta_-$ and $\beta_+$ denote the smallest and largest eigenvalues of $W$, respectively. Then, the condition
\begin{equation}
\|W-2^{K}(1-\alpha)\mathbb{I}_n\|-\|W\|+2\theta\leq 2^{K}\alpha
\label{eq:condition31}
\end{equation}
is satisfied if and only if one of the following conditions holds:
\begin{enumerate}
\item\label{prop:i} $\beta_++\beta_-\leq 0$ and $\theta\leq 2^{K-1}(2\alpha-1)$;
\item\label{prop:ii} $0\leq\beta_++\beta_-\leq 2^{K+1}(1-\alpha)$ and $2\theta\leq \beta_++\beta_-+2^{K}(2\alpha-1)$;
\item\label{prop:iii} $2^{K+1}(1-\alpha)\leq \beta_++\beta_-$ and $\theta\leq 2^{K-1}$.
\end{enumerate}
\label{prop:condition31}
\end{proposition}
\begin{proof}
Let $\alpha\in[1/2,1]$. Since $W\in\mathcal{S}_n$, we have, $\|W\|=\max\{\beta_+,-\beta_-\}$, and
\begin{equation}
\|W-2^{K}(1-\alpha)\mathbb{I}_n\|=\max\left\{\beta_+-2^{K}(1-\alpha),-\beta_-+2^{K}(1-\alpha)\right\}.
\label{eq:proof0}
\end{equation}
Three different cases arise that we review below.
\begin{enumerate}

\item[(i)] If $\beta_++\beta_-\leq 0$ then $\|W\|=-\beta_-$ and 
\begin{equation}
\beta_+-2^{K}(1-\alpha)\leq -\beta_-+2^{K}(1-\alpha).
\label{eq:proof1}
\end{equation}
From \eqref{eq:proof0} and \eqref{eq:proof1}, we deduce that $\|W-2^{K}(1-\alpha)\mathbb{I}_n\|=-\beta_-+2^{K}(1-\alpha)$. Replacing $\|W\|$ and $\|W-2^{K}(1-\alpha)\mathbb{I}_n\|$ by their value in \eqref{eq:condition31} leads to Proposition~\ref{prop:condition31}\eqref{prop:i}.

\item[(ii)] If $0\leq\beta_++\beta_-\leq 2^{K+1}(1-\alpha)$ then $\|W\|=\beta_+$ and \eqref{eq:proof1} is satisfied. Hence, $\|W-2^{K}(1-\alpha)\mathbb{I}_n\|=-\beta_-+2^{K}(1-\alpha)$. Replacing $\|W\|$ and $\|W-2^{K}(1-\alpha)\mathbb{I}_n\|$ by their value in \eqref{eq:condition31} leads to Proposition~\ref{prop:condition31}\eqref{prop:ii}.

\item[(iii)] If $2^{K+1}(1-\alpha)\leq \beta_++\beta_-$ then $\|W\|=\beta_+$ and
\begin{equation}
\beta_+-2^{K}(1-\alpha)\geq -\beta_-+2^{K}(1-\alpha).
\label{eq:proof2}
\end{equation}
From \eqref{eq:proof0} and \eqref{eq:proof2}, we deduce that $\|W-2^{K}(1-\alpha)\mathbb{I}_n\|=\beta_+-2^{K}(1-\alpha)$. Replacing $\|W\|$ and $\|W-2^{K}(1-\alpha)\mathbb{I}_n\|$ by their value in \eqref{eq:condition31} leads to Proposition~\ref{prop:condition31}\eqref{prop:iii}, which completes the proof.
\end{enumerate}
\end{proof}

%%%%%%%%%%%%%%%%%%%%%%%%%%%%%%%%%%%%%%%%%%%%%%%%%%%%%%%%%%%%%%
\subsection{Averaged operator}
The notion of nonexpansiveness, whose definition is recalled below, plays a central role in the analysis of the robustness of nonlinear operators. 
We recall that $T:\mathbb{R}^n\rightarrow \mathbb{R}^n$ is nonexpansive if it is Lipschitz continuous with constant $1$, i.e., 
\begin{equation}
(\forall x\in \mathbb{R}^n)(\forall y\in \mathbb{R}^n)~~\|T(x)-T(y)\|\leq \|x-y\|.
\end{equation}

In the present study we make use of the notion of averaged operator \cite{bauschke2017convex}, which is stronger than nonexpansiveness. $T$ is $\alpha$--averaged with $\alpha\in[0,1]$, if there exists a nonexpansive operator $R:\mathbb{R}^n\rightarrow \mathbb{R}^n$ such that $T=(1-\alpha)I_n+\alpha R$, where $I_n$ denotes the identity operator of $\mathbb{R}^n$.

\noindent The following property provides an upper bound of the effect of an input perturbation, which depends on the averageness constant $\alpha$. In particular, the smaller $\alpha$ is, the more stable the operator is.

\begin{proposition}{\rm \cite[Remark~4.34, Proposition~4.35]{bauschke2017convex}} Let $T:\mathbb{R}^n\rightarrow \mathbb{R}^n$. 
\begin{enumerate}
\item[(i)] If $T$ is averaged, then it is nonexpansive.
\item[(ii)] Let $\alpha\in \ ]0,1]$. $T$ is $\alpha$--averaged if and only if for every $x\in \mathbb{R}^n$ and $y\in \mathbb{R}^n$,
\begin{equation}
\|T(x)-T(y)\|^2\leq \|x-y\|^2  -\frac{1-\alpha}{\alpha}\|(I_n-T)(x)-(I_n-T)(y)\|^2.
\end{equation} 
\end{enumerate}
\label{prop:averaged_bound}
\end{proposition}

\subsection{Robustness of iRestNet to an input perturbation}

Let us consider problem~\eqref{pb:net_stability}, where we assume additionally that $H^\top H$ and $D^\top D$ are diagonalizable in a same basis denoted $\mathcal{P}$. The latter is satisfied for instance if $H$ and $D$ are the results of cyclic convolutive operators.
Theorem~\ref{thm:averaged} below gives sufficient conditions under which the proposed network applied to problem~\eqref{pb:net_stability} is averaged. 

\begin{theorem}
Let $\alpha\in[1/2,1]$, $(W_k,b_k,R_k)_{0\leq k\leq K-1}$ be defined by \eqref{eq:weights_bias_act}, and $(\theta_{k})_{-1\le k \le K-1}$ be defined as in Proposition~\ref{prop:theta}. Let $\beta_-$ and $\beta_+$ be the smallest and largest eigenvalues of $W=W_{K-1}\circ\cdots\circ W_0$, respectively. For every $p \in \{1,\ldots, n\}$ and every $k\in\{0,\ldots,K-1\}$, let $\beta_k^{(p)}=1-\gamma_k\left(\beta_H^{(p)}+\lambda_k\beta_D^{(p)}\right)$,
where $\beta_H^{(p)}$ and $\beta_D^{(p)}$ denote the $p^{\text{th}}$ eigenvalue of $H^\top H$ and $D^\top D$ in $\mathcal{P}$, respectively. Then, $\beta_-$, $\beta_+$, and $(\forall k\in\{0,\ldots,K-1\})$ $\theta_{k}$ can be computed as follows:

\begin{equation}
\beta_-=\min_{1\leq p\leq n}\prod_{k=0}^{K-1}\beta_k^{(p)}
,~
\beta_+=\max_{1\leq p\leq n}\prod_{k=0}^{K-1}\beta_k^{(p)}
~
\text{and}~~\theta_{k}=\sum_{l=0}^k \theta_{l-1}\max_{1\leq q_l\leq n}\left|\beta_k^{(q_l)}\ldots\beta_l^{(q_l)}\right|.
\label{eq:thm1_eq1}
\end{equation}
In addition, if one of the following conditions is satisfied
\begin{enumerate}
\item[(i)]\label{thm:cond1} $\beta_++\beta_-\leq 0$ and $\theta_{K-1}\leq 2^{K-1}(2\alpha-1)$;
\item[(ii)]\label{thm:cond2} $0\leq\beta_++\beta_-\leq 2^{K+1}(1-\alpha)$ and $2\theta_{K-1}\leq \beta_++\beta_-+2^{K}(2\alpha-1)$;
\item[(iii)]\label{thm:cond3} $2^{K+1}(1-\alpha)\leq \beta_++\beta_-$ and $\theta_{K-1}\leq 2^{K-1}$,
\end{enumerate} 
then the operator $R_{K-1}\circ(W_{K-1}\cdot+b_{K-1})\circ\cdots\circ R_0\circ(W_0\cdot+b_0)$ is $\alpha$--averaged. 
\label{thm:averaged}
\end{theorem}
\begin{proof}
If $H^\top H$ and $D^\top D$ are diagonalizable in the same basis then $W\in\mathcal{S}_n$, which, combined with Proposition~\ref{prop:theta}, leads to \eqref{eq:thm1_eq1}. If one of the conditions~(i)--(iii) is satisfied, then we deduce from Proposition~\ref{prop:condition31} that $W$ satisfies \cite[Proposition~3.6(iii)]{combettes2018deep}. and \cite[Condition~3.1]{combettes2018deep}. In addition, for every $k\in\{0,\ldots,K-1\}$, $R_k(\cdot+b_k)$ is firmly nonexpansive \cite[Proposition~12.28]{bauschke2017convex}. Finally, \cite[Theorem~3.8]{combettes2018deep} completes the proof.  
\end{proof}

\noindent The conditions provided by Theorem~\ref{thm:averaged} can be easily checked using \eqref{eq:thm1_eq1}. Theorem~\ref{thm:averaged} provides a framework under which iRestNet is robust to a perturbation of its input: the upper bound of the output perturbation can then be derived from Proposition~\ref{prop:averaged_bound}. 
%One can also note that the proposed activation functions $(R_k)_{0\leq k\leq K-1}$, defined in \eqref{eq:weights_bias_act}, are also individually robust to a perturbation by nonexpansiveness of the proximity operator. 
%The ability to control the robustness of the network appears as a very desirable property in many application fields, such as medical image processing.

\section{Numerical experiments}
\label{sec:exp}

In this section, we present numerical experiments on a set of problems of image restoration, demonstrating that in many cases the proposed approach yields a better reconstruction quality than standard variational and machine learning methods.

%%%%%%%%%%%%%%%%%%%%%%%%%%%%%%%%%%%%%%%%%%%%%%%%%%%%%%%%%%%%
\subsection{Problem formulation}
We consider the non-blind color image deblurring problem, whose degradation model reads
\begin{equation}
y=H\overline{x}+\omega,
\end{equation}
where $n$ is the number of pixels, $y=(y^{(j)})_{1\leq j\leq 3}\in\mathbb{R}^{3n}$ is the blurred RGB image, $\overline{x}=(\overline{x}^{(j)})_{1\leq j\leq 3}\in\mathbb{R}^{3n}$ is the ground-truth, $H \in {\mathbb{R}}^{3n\times 3n}$ is a linear operator that models the circular convolution of a known blur kernel with each channel of the color image, and $\omega\in\mathbb{R}^{3n}$ is a realization of an additive white Gaussian noise with standard deviation $\sigma$. An estimate of $\overline{x}$ can be derived from the following penalized formulation, which includes a smoothed total variation regularization,
\begin{equation}
\minimize_{x\in\mathcal{C}} ~~\frac{1}{2}\|Hx-y\|^2+\lambda\sum_{i=1}^{3n} \sqrt{\frac{\left(D_{\text{v}}x\right)_i^2+\left(D_{\text{h}}x\right)_i^2}{\delta^2}+1},
\label{pb:deblurring}
\end{equation}
\noindent where the feasible set ${\mathcal{C}}$ is the hypercube $[x_{\rm min},x_{\rm max}]^{3n}$, $x_{\rm min}$ and $x_{\rm max}$ are a lower and an upper bound on the pixel intensity, respectively, $D_{\text{v}}\in\mathbb{R}^{3n\times 3n}$ and $D_{\text{h}}\in\mathbb{R}^{3n\times 3n}$ are the vertical and horizontal gradient operators, respectively, $\delta>0$ is a smoothing parameter and $\lambda>0$ is the regularization parameter. Here, $x_{\rm min}=0$, $x_{\rm max}=1$ and we set $\delta=0.01$ in all experiments, which appears as an appropriate order of magnitude. To find this value for $\delta$, we solved Problem~{\eqref{pb:deblurring}} for a small set of images of the database and used the simplex method to find the best values for $\delta$ and $\lambda$ in terms of image quality. It is worth noting that the value for $\delta$ has not been fine-tuned, but that the proposed architecture could also be easily modified to include the inference of $\delta$. The update $\mathcal{A}$, defined in \eqref{eq:mathcalA}, is derived from \eqref{pb:deblurring}, and is unfolded over $K$ iterations, as it is described in Section~\ref{sec:net}.
The bound constraints in problem~\eqref{pb:deblurring} fall under the framework studied in Section~\ref{sec:hyperslab}, which provides us with the expression for the proximity operator of the barrier and its gradient.

%%%%%%%%%%%%%%%%%%%%%%%%%%%%%%%%%%%%%%%%%%%%%%%%%%%%%%%%%%%%
\subsection{Network characteristics}
The tuning of the number of unfolded iterations $K$ must achieve a compromise between training time, memory requirement, and performance.  In order to determine a suitable setting for $K$, we trained networks with different numbers of layers and increased the number of layers until the performance of the network did not improve significantly. Using this procedure, the depth of iRestNet is taken equal to $K=40$.
Regarding the hidden structures $(\mathcal{L}_k^{(\lambda)})_{0\leq k\leq K-1}$, which estimate the regularization parameter, they are chosen in view of the regularization function used in problem~\eqref{pb:deblurring} and have the following expression,
\begin{equation}
(\forall k\in\{0,\ldots, K-1\})~~ \lambda_k=\mathcal{L}_k^{(\lambda)}\left(x_k\right)=\frac{\mathrm{Softplus}\left(b_k\right)\widehat{\sigma}(y)}{\eta(x_k)+\mathrm{Softplus}\left(c_k\right)},
\label{eq:reg_TV}
\end{equation}
where $(b_k,c_k)$ is a pair of scalars learned by the network, $\eta(x_k)$ is the standard deviation of the concatenated spatial gradients of $x_k$, $[(D_{\rm v} x_k)^\top (D_{\rm h} x_k)^\top]$, and $\widehat{\sigma}(y)$ is an approximation of the noise level in the blurred image. The noise level is estimated as in \cite[Section~11.3.1]{mallat1999wavelet}

\begin{equation}
\widehat{\sigma}(y)=\text{median}(|W_{\rm H}y|)/0.6745,
\end{equation}
where $|W_{\rm H}y|$ is the vector gathering the absolute value of the diagonal coefficients of the first level Haar wavelet decomposition of $y$. It is worth noticing that the proposed architecture does not require any prior knowledge about the noise level, in particular the noise standard deviation does not have to be the same for all input images.  
\\The architecture of the post-processing layer $\mathcal{L}_{\rm pp}$ is inspired from \cite{zhang2017learning}: it is made of 9 convolutional layers with filters of size $3\times 3$. The dilation factor changes from one layer to another, so as to widen the receptive field without creating memory issues. There is little correlation between the artifacts that remain in the image after going through the 40 blocks of iRestNet and the ground-truth image. Hence, it is easier for the network to learn the residual mapping instead of the image itself since pushing the residual to zero is easier than fitting an identity mapping by a stack of layers~\cite{zhang2017beyond,zhang2017learning,he2016deep}. Therefore, we add a skip connection between the input of $\mathcal{L}_{\rm pp}$ and its output. Finally, a ReLU activation function is used after each convolution, the final activation function is chosen as the Sigmoid function, and residual learning is combined with batch normalization, a technique which is widely used in deep learning to accelerate and stabilize the training process \cite{zhang2017learning}. The final architecture of $\mathcal{L}_{\rm pp}$ can be found in Figure~\ref{fig:cnnpp}.

\begin{figure}
\centering
\hspace{-0.8cm}
\includegraphics[width=1.02\textwidth]{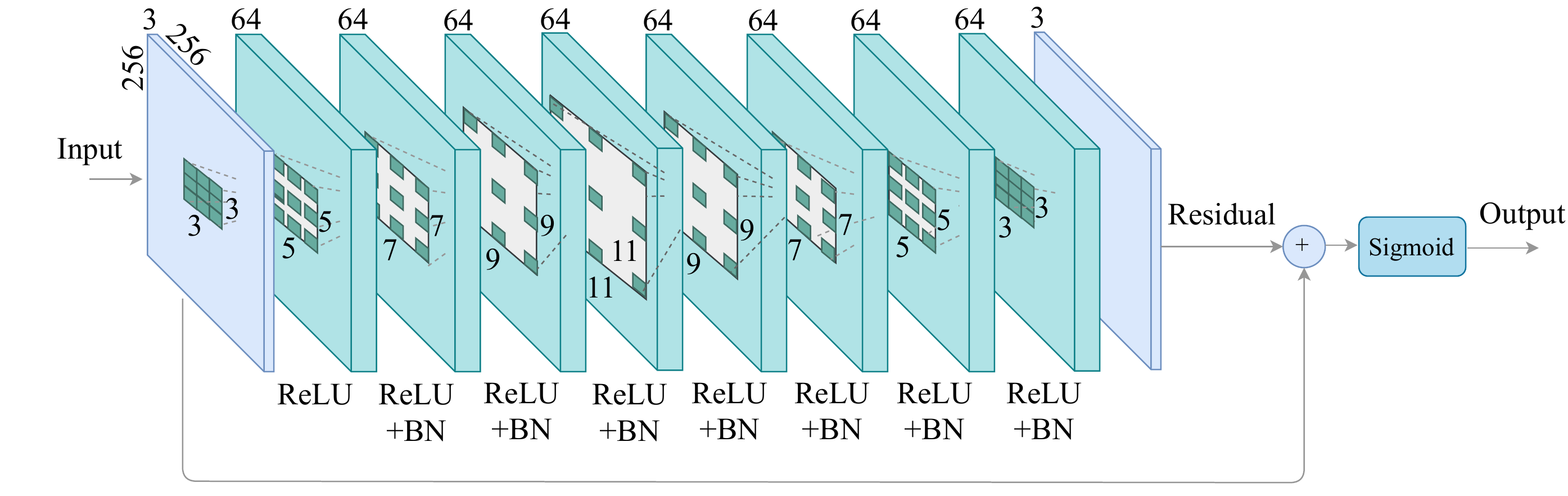}
\caption{Architecture of $\mathcal{L}_{\rm pp}$. BN: batch normalization.}
\label{fig:cnnpp}
\end{figure}

%%%%%%%%%%%%%%%%%%%%%%%%%%%%%%%%%%%%%%%%%%%%%%%%%%%%%%%%%%%%
\subsection{Dataset and experimental settings}

The training set is made of 1200 RGB images: 200 images stem from the Berkeley segmentation (BSD500) training set, while the remaining 1000 images are taken from the COCO training set. We use the BSD500 validation set, which is made of 100 images, to monitor the training and check if there is overfitting. The performance of the proposed method is evaluated on two different test sets: the BSD500 test set, which is made of 200 RGB images, and the Flickr30 test set used in \cite{xu2014deep}, which is made of 30 RGB images.
The test images have been center-cropped using a window of size $256\times 256$. Blurry images are produced using the following $25\times 25$ blur kernels and noise levels:

\begin{enumerate}
\item[-] A Gaussian kernel, which models atmospheric turbulence, with a standard deviation of 1.6 pixels, and a Gaussian noise standard deviation of $\sigma=0.008$. This configuration is denoted as GaussianA. To evaluate the robustness of the proposed method with respect to the noise level, the same kernel is used with a Gaussian noise whose standard deviation is uniformly distributed between 0.01 and 0.05. The latter is denoted as GaussianB.
\item[-] The Gaussian kernel with a standard deviation of 3 pixels, and a Gaussian noise standard deviation of $\sigma=0.04$, denoted as GaussianC.
\item[-] The eighth and third motion test kernels from \cite{levin2009understanding}, which are real-world camera shake kernels, with a Gaussian noise standard deviation of $\sigma=0.01$. These settings are denoted as MotionA and MotionB, respectively. 
\item[-]The square uniform kernel of size $7\times 7$, with a Gaussian noise standard deviation of $\sigma=0.01$. This configuration is referred to as Square. 
\end{enumerate}

%%%%%%%%%%%%%%%%%%%%%%%%%%%%%%%%%%%%%%%%%%%%%%%%%%%%%%%%%%%%
\subsection{Training}
\label{subsec:training}
For each degradation model, one iRestNet network is trained. We use a greedy approach for training the first 30 layers. For $\mathcal{L}_0$, a minibatch of 10 images is selected at every iteration, randomly cropped using a window of size $256\times 256$, blurred with the given kernel, and degraded with Gaussian noise; the training of $\mathcal{L}_0$ stops after a fixed number of epochs. Then, for each image of the training set, a random crop of size $256\times 256$ is selected, blurred, corrupted with noise and passed through $\mathcal{L}_0$, the output is saved and used as an input to train $\mathcal{L}_1$. When the training of $\mathcal{L}_1$ is complete, its output is used to train the next layer, etc... This training strategy is chosen with regards to its low memory requirement: the number of layers is not limited by the hardware.
The rest of the network, $\mathcal{L}_{\rm pp}\circ\mathcal{L}_{39}\circ\ldots\circ\mathcal{L}_{30}$, is trained as one block and the learning rate is multiplied by $0.9$ every $50$ epochs. To accelerate the training, for every $k\in\{1,\ldots,K-1\}$, the weights of $\mathcal{L}_k$ are initialized with those of $\mathcal{L}_{k-1}$. Detailed information about learning rates and number of epochs can be found in Table~\ref{tab:lr_epochs} below.
\begin{table}[h!]
\centering
\setlength\tabcolsep{2pt}
\begin{tabular}{rcccccc}
\toprule
 & GaussianA & GaussianB & GaussianC & MotionA & MotionB & Square\\
 \midrule
Rates & ($0.01$,$0.001$) & ($0.01$,$0.001$) & ($0.001$,$0.001$) & ($0.01$,$0.002$)& ($0.01$,$0.001$) & ($0.01$,$0.005$)\\
 Epochs & (40,393) & (40,340) & (40,300) & (40,1200) & (40,1250) & (40,740)\\
 \bottomrule
\end{tabular}
\caption{Training information. First row: initial learning rates, second row: number of epochs. For every couple, the first and second numbers correspond to the training of $(\mathcal{L}_k)_{0\leq k\leq 29}$ and $\mathcal{L}_{\rm pp}\circ\mathcal{L}_{39}\circ\ldots\circ\mathcal{L}_{30}$, respectively.}
\label{tab:lr_epochs}
\end{table}

\noindent The validation set is used to monitor this last step of the training. In particular, the configuration of network parameters that gives the best performance on the validation set during the training is the one saved and used for the tests.
Note that for the first 30 layers, after each layer the quality of the restored training images should improve. This property comes from the training strategy, it is not encoded in the network: if memory was not an issue, then iRestNet could be trained in an end-to-end fashion.
\\
We use the Adam optimizer \cite{kingma2014adam} to minimize the training loss, which is taken as the negative of the structural similarity measure (SSIM) \cite{wang2004image} defined below
%which is defined as follows.
%\begin{equation}
%\mathrm{LOSS}\left(\{x^{(i)}\}_{1\leq i\leq q}\right) = \sum_{i=1}^q -\mathrm{SSIM}\left(x^{(i)},\overline{x}^{(i)}\right)
%\end{equation}
%Where $q$ is the size of the minibatch, $(\forall i\{1,\cdots,q\})$ $x^{(i)}$ and $\overline{x}^{(i)}$ are the $i^{\rm th}$ restored and ground-truth images of the minibatch, respectively, and SSIM is the structural similarity measure \cite{wang2004image}, which is defined as follows:
\begin{equation}
\mathrm{SSIM}(x,\overline{x})=\frac{(2\mu_x\mu_{\overline{x}}+c_1)(2\sigma_x\sigma_{\overline{x}}+c_2)(2\mathrm{cov}_{x\overline{x}}+c_3)}{(\mu_x^2+\mu_{\overline{x}}^2+c_1)(\sigma_x^2+\sigma_{\overline{x}}^2+c_2)(\sigma_x\sigma_{\overline{x}}+c_3)},
\end{equation}
where $\overline{x}$ is the ground truth, $x$ is the restored image, $(\mu_x,\sigma_x)$ and $(\mu_{\overline{x}},\sigma_{\overline{x}})$ are mean and standard deviation of $x$ and $\overline{x}$, respectively, $\mathrm{cov}_{x\overline{x}}$ is the cross--covariance of $x$ and $\overline{x}$, and $c_1$, $c_2$ and $c_3$ are constants.
As explained in \cite{wang2004image}, the SSIM is a good measure of perceived visual quality, since it is based on how the human eye extracts structural information from an image. Hence, it is more discriminative with regards to artifacts than the mean square error for instance. The gradient of the SSIM loss with respect to the trainable parameters of the network is computed using the code {available online \footnote{\url{https://github.com/Po-Hsun-Su/pytorch-ssim}} and based on} \cite{wang2004image}, the chain rule, automatic differentiation~\cite{paszke2017autodiff}, and the expression given in Section~\ref{sec:hyperslab} for the derivatives of the barrier proximity operator. 
\\
Codes are implemented in Pytorch {and are available online \footnote{\url{https://github.com/mccorbineau/iRestNet}}}. Some hidden layers {in the post-processing part make use of ReLU, which is not differentiable everywhere}. Since this nondifferentiability happens only at specific points for which the left and right derivatives are well--defined, Pytorch can handle it as explained in \cite{Goodfellow-et-al-2016}.
All trainings are conducted using a GeForce GTX 1080 GPU or a Tesla V100 GPU. The training, which can be performed off-line, takes approximately 3 to 4 days for each blur kernel, while the time taken per test image is only about 1.4 sec on a GeForce GTX 1080 GPU.

%%%%%%%%%%%%%%%%%%%%%%%%%%%%%%%%%%%%%%%%%%%%%%%%%%%%%%%%%%%%
\subsection{Evaluation metrics and competitors}
The restoration is evaluated in terms of the SSIM metric.
The reconstruction given by the proposed approach is compared with a solution to problem~(\ref{pb:deblurring}) obtained using the projected gradient algorithm \cite{iusem03pgd}.
%for pre-determined values of $\lambda$ and $\delta$. 
For every blurred image, the pair $(\lambda,\delta)$ which leads to the best SSIM is selected using the simplex method. The solution given by this variational approach is referred to as VAR. The latter is an unrealistic scenario since it assumes that there is a perfect estimator of the error, but it gives an upper bound on the image quality that one can expect by solving (\ref{pb:deblurring}). We also use the following deep learning image restoration methods for comparison: EPLL \cite{zoran2011learning} {and MLP \cite{schuler2013machine}. 
Finally, we include comparisons with three unfolded-based methods, namely IRCNN~\cite{zhang2017learning}, where an empirical algorithm\footnote{{In \cite{zhang2017learning}, this algorithm is improperly called half-quadratic splitting, but it does not correspond to usual half-quadratic optimization methods described for instance in \cite{allain2006global}. Actually, the algorithm unfolded in \cite{zhang2017learning} can be interpreted as a preconditioned forward-backward algorithm.}} derived from an augmented Lagrangian formulation is unfolded over 30 iterations and a CNN is used as a denoiser to update the splitting variable, FCNN \cite{zhang2017learningiterative}, where the authors unfold the same algorithm as in the previous method} and use a network to learn an effective regularization function, and the method from \cite{meinhardt2017learning}, which is referred to as PDHG, where the authors perform a maximum of 30 iterations of a primal dual hybrid gradient algorithm and the proximity operator of the second regularization function is replaced by a neural network. 
\\
For FCNN, we use the code that is available online, in which the authors provide a model that has only been trained for motion blurs. Hence, for a fair comparison, we only provide the results of FCNN on MotionA and MotionB, and we specify that this method is not applicable (n/a) to the other configurations. Similarly, for MLP and PDHG, the authors do not provide models that were trained specifically for MotionB and Square, so we do not test these methods on these two configurations.
\\
Since MLP, EPLL and IRCNN require the knowledge of the noise level, for the GaussianB degradation model, we make use of the estimation of the noise standard deviation given by the method in \cite{ramadhan2017image}. In addition, since some comparison methods, like EPLL for instance, do not estimate well the borders of the images, the SSIM index is computed excluding a 6-pixel-wide frame for all images and all tested methods.      

%%%%%%%%%%%%%%%%%%%%%%%%%%%%%%%%%%%%%%%%%%%%%%%%%%%%%%%%%%%%
\subsection{Results and discussion}

The average SSIM obtained with the different methods for the various blur kernels and noise levels on the BSD500 test set can be found in Table~\ref{tab:res_BSD500}. The mean SSIM achieved with iRestNet on this test set is greater than those obtained with the other methods for all degradation models except MotionA. For this kernel, the average SSIM achieved with iRestNet is the second highest value after IRCNN, which appears as the most competitive method. IRCNN involves two steps: first, a Wiener filter is applied to the blurred image, then, a neural network is used to predict the residual and denoise the image. These two steps are repeated 30 times, for 30 different manually tuned regularization parameters. In contrast, iRestNet does not require any tuning from the user regarding the regularization parameters during training. For completeness, the SSIM of all images of the BSD500 test set are plotted in Figure~\ref{fig:SSIM_sorted} for the 6 different degradation models. As one can see, iRestNet performs well in terms of SSIM on most of the images.
\\ 
\begin{table}[h!]
\setlength\tabcolsep{4pt}
\centering
\begin{tabular}{lcccccc}
\toprule
                              & GaussianA    & GaussianB    & GaussianC    &  MotionA      & MotionB      & Square      \\
\midrule
Blurred                       & 0.676        & 0.526        & 0.326        & 0.383        & 0.549        & 0.544        \\
VAR                           & 0.804        & 0.723        & 0.587        & 0.819        & 0.829        & 0.756        \\ 
EPLL \cite{zoran2011learning} & 0.800        & 0.708        & 0.565        & 0.816        & 0.839        & 0.755        \\
MLP \cite{schuler2013machine} & 0.821        & 0.734        & 0.608        & 0.854        & n/a        & n/a        \\
PDHG \cite{meinhardt2017learning}
                              & 0.796        & 0.716        & 0.563        & 0.801        & n/a        & n/a        \\ 
IRCNN \cite{zhang2017learning}& 0.841        & 0.768        & 0.619        &\textbf{0.902}& 0.907        & 0.834        \\
FCNN \cite{zhang2017learningiterative}
                              & n/a        & n/a        & n/a        & 0.794        & 0.847        & n/a        \\
iRestNet                      &\textbf{0.853}&\textbf{0.787}&\textbf{0.641}& 0.898        &\textbf{0.910}&\textbf{0.840}\\
\bottomrule
\end{tabular}
\caption{SSIM results on the BSD500 test set.}
\label{tab:res_BSD500}
\end{table}

\begin{figure}
\setlength\tabcolsep{4pt}
\centering
\begin{tabular}{cc}
 \includegraphics[width=0.47\textwidth]{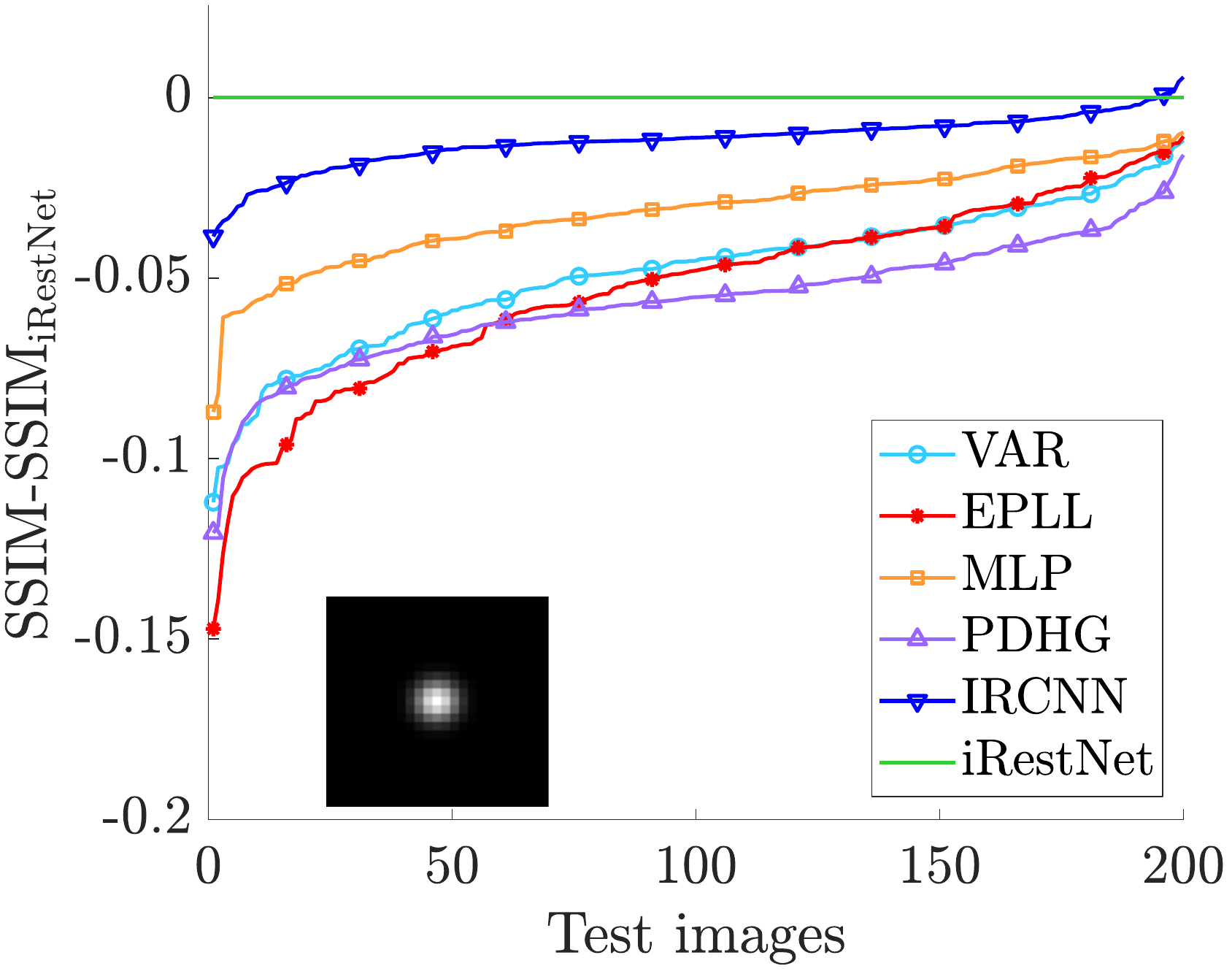} & \includegraphics[width=0.47\textwidth]{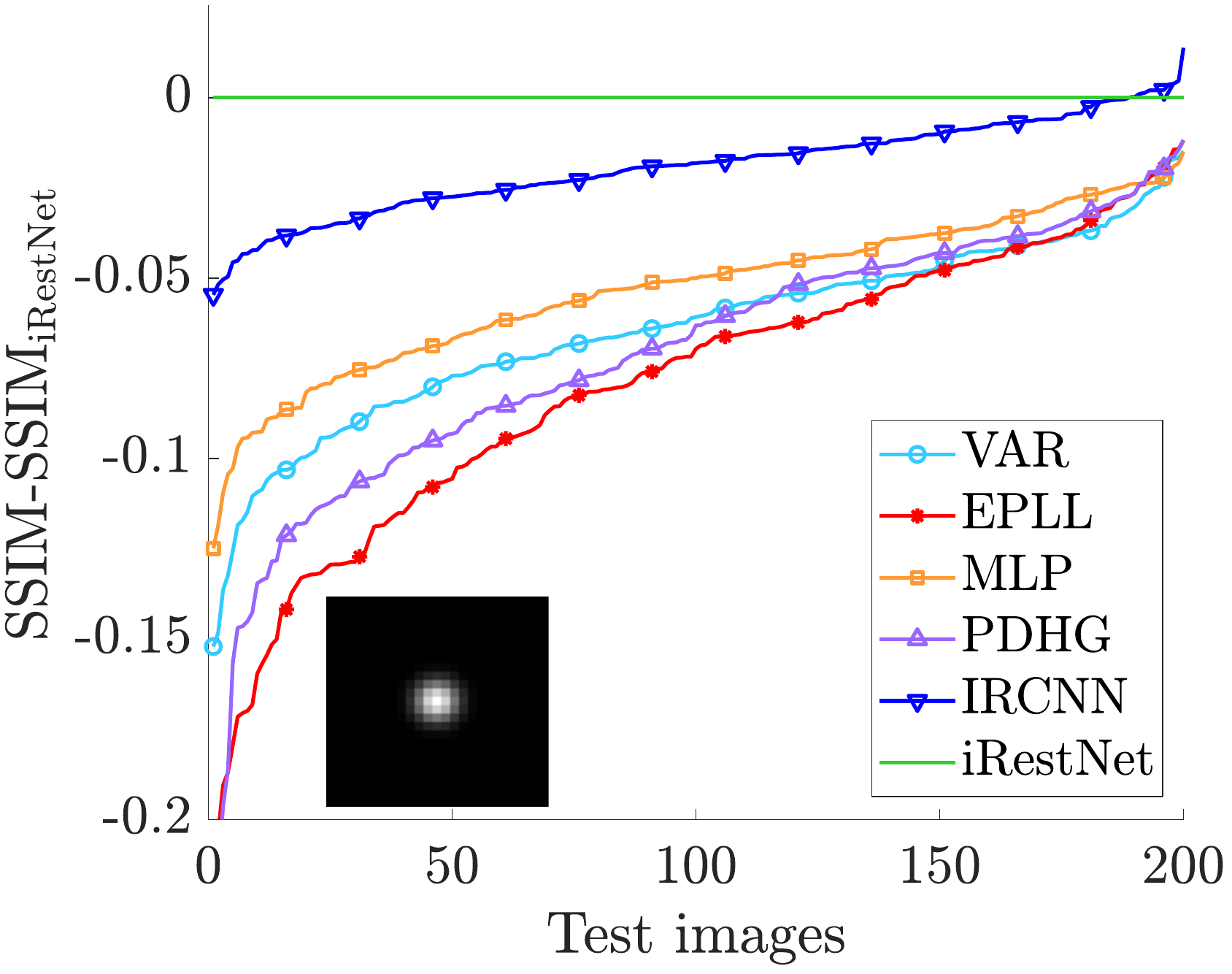}\\
(a) & (b)\\
 \includegraphics[width=0.47\textwidth]{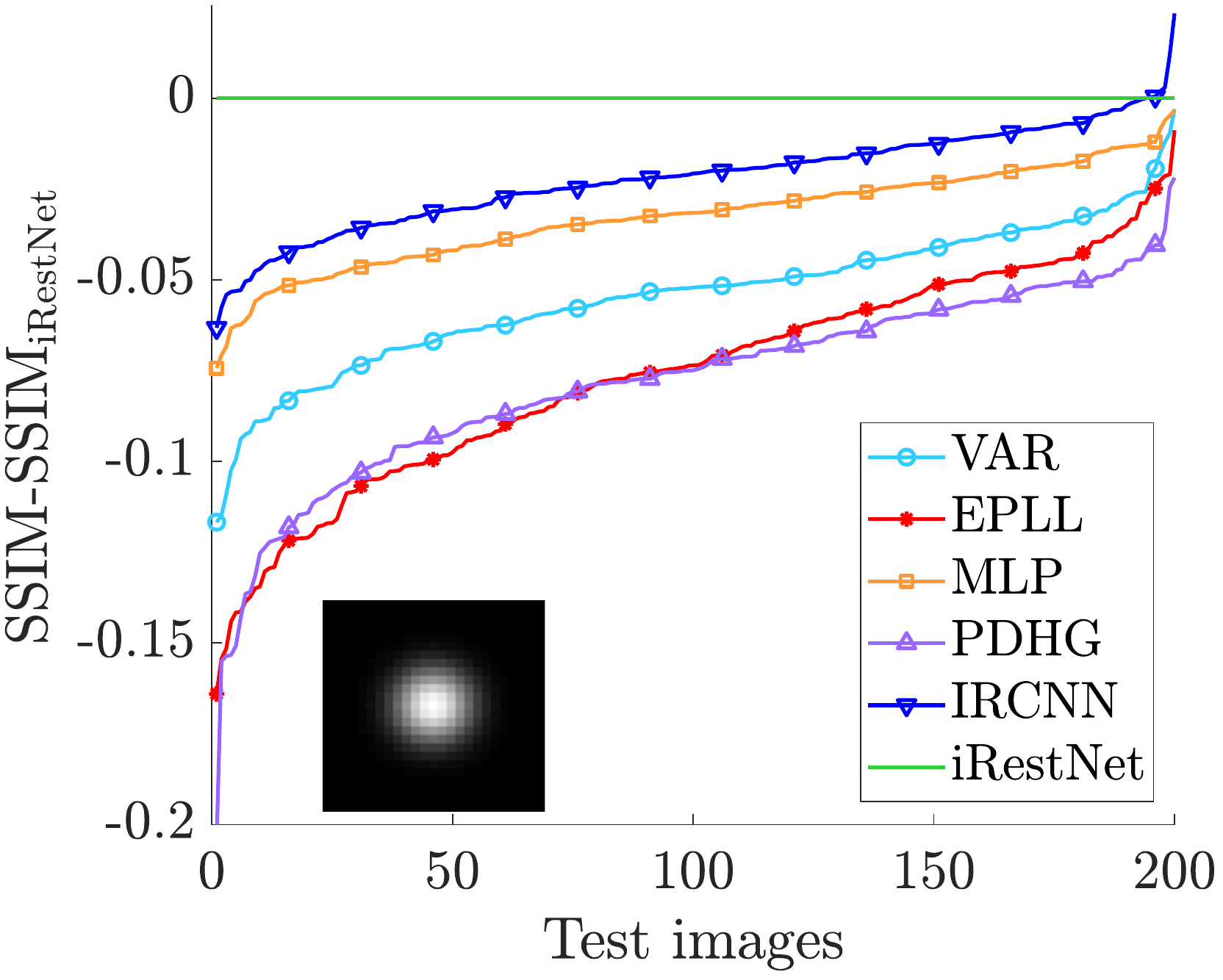}   & \includegraphics[width=0.47\textwidth]{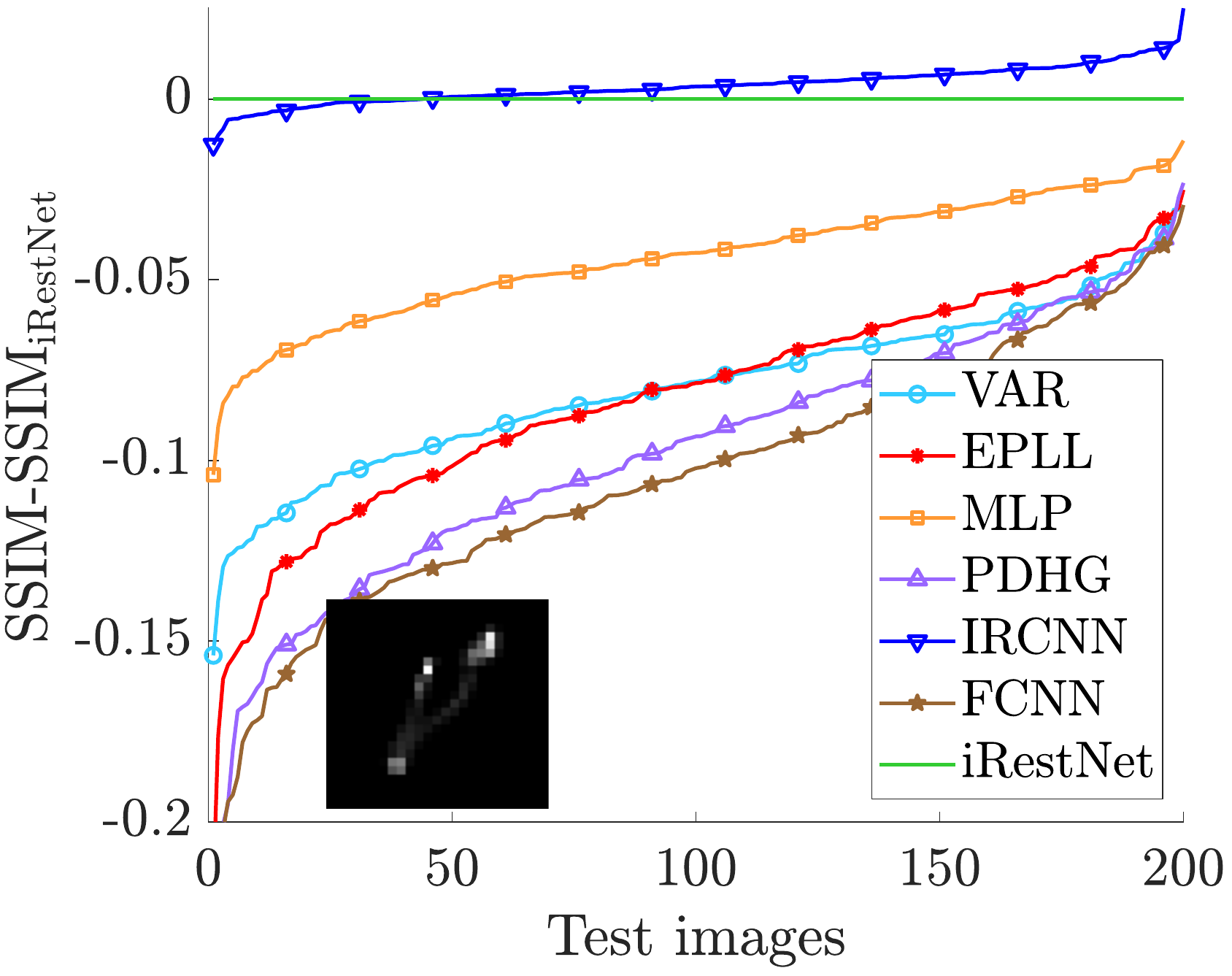}   \\
(c) & (d)\\
 \includegraphics[width=0.47\textwidth]{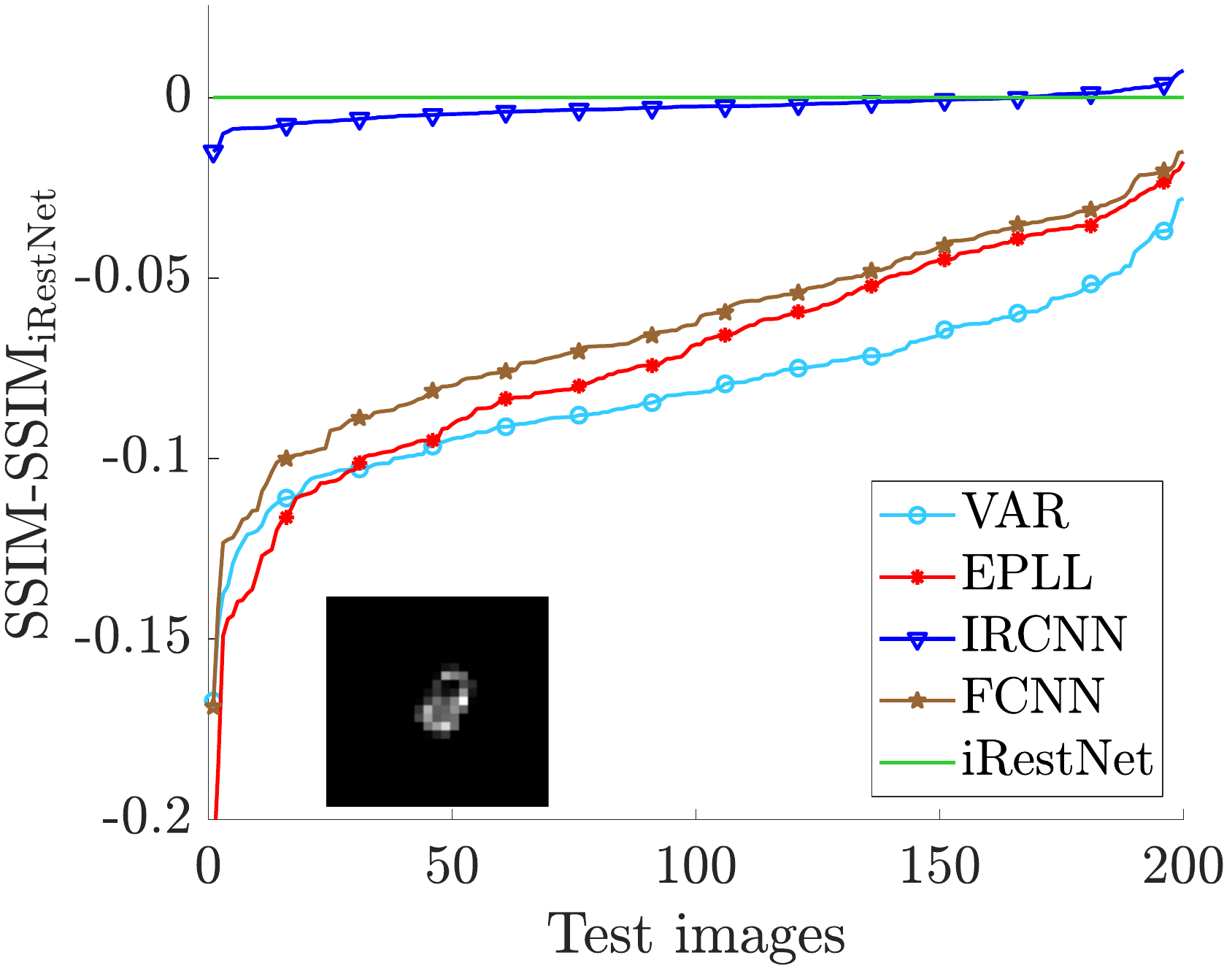}   &  \includegraphics[width=0.47\textwidth]{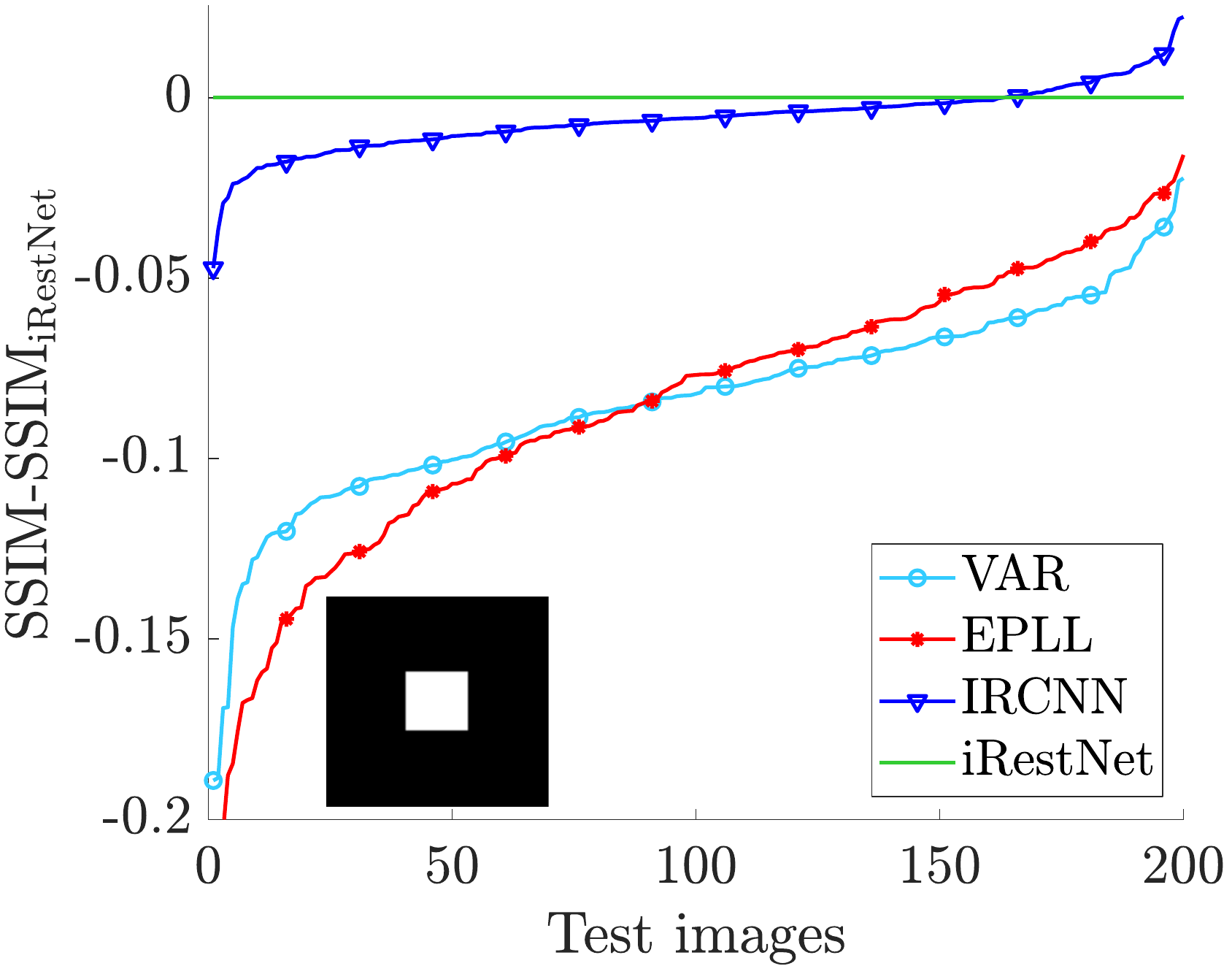}  \\
(e) & (f)\\
\end{tabular}
\caption{Sorted improvement of iRestNet with regards to other methods on the BSD500 test set using the SSIM metric: a negative value indicates a better performance of iRestNet. (a): GaussianA, (b): GaussianB, (c): GaussianC, (d): MotionA, (e): MotionB, (f): Square.}
\label{fig:SSIM_sorted}
\end{figure}

\noindent Since no image was taken from Flickr for training iRestNet, the results on the Flickr30 test set show how well the performance of the trained networks are transferable on test sets with statistics that are different from those of the training set. Table~\ref{tab:res_Flickr30} contains the average SSIM obtained with the different methods on the Flickr30 test set. Similarly to the BSD500 test set, iRestNet compares favorably with the other approaches on the Flickr30 test set.
\begin{table}[h!]
\setlength\tabcolsep{4pt}
\centering
\begin{tabular}{lcccccc}
\toprule
                               & GaussianA    & GaussianB     & GaussianC         & MotionA      & MotionB      & Square       \\
\midrule 
Blurred                        & 0.723        & 0.545         & 0.355             & 0.376        & 0.590        & 0.579        \\
VAR                            & 0.857        & 0.776         & 0.639             & 0.856        & 0.869        & 0.818        \\ 
EPLL \cite{zoran2011learning}  & 0.860        & 0.770         & 0.616             & 0.857        & 0.887        & 0.827        \\
MLP \cite{schuler2013machine}  & 0.874        & 0.798         & 0.668             & 0.891        & n/a        & n/a        \\
PDHG \cite{meinhardt2017learning}
                               & 0.853        & 0.781         & 0.623             & 0.855        & n/a        & n/a        \\
IRCNN \cite{zhang2017learning} & 0.885        & 0.819         & 0.676             &\textbf{0.927}&\textbf{0.930}&\textbf{0.886}\\
FCNN \cite{zhang2017learningiterative}
                               & n/a        & n/a         & n/a             & 0.801        & 0.890        & n/a        \\
iRestNet                       &\textbf{0.892}&\textbf{0.833} &\textbf{0.696}     & 0.919        &\textbf{0.930}&\textbf{0.886}\\
\bottomrule
\end{tabular}
\caption{SSIM results on the Flickr30 test set.}
\label{tab:res_Flickr30}
\end{table}

Examples of visual results obtained with the different methods can be found in Figures~\ref{fig:gaussianB_image} and \ref{fig:square_image} for two images from the BSD500 test set and the blur kernels GaussianB and Square, respectively. We also provide the results obtained for one image from the Flickr30 test set that has been degraded with MotionB. As one can see from inspecting these pictures, details from the snake's and caterpillar's skin patterns are better retrieved with iRestNet, which provides more visually-satisfactory results than competitors. Similarly, on Figure~\ref{fig:motionB_image}, competitors tend to smooth too much the details on the leaves as it can be seen in the top left-hand corner. Regarding Figure~\ref{fig:gaussianB_image}, which belongs to the test set with a level-varying noise, it is worth noting that, on the result obtained with the proposed method, the green background is free from artifacts, which is not the case for the other methods, in particular for PDHG and IRCNN. This suggests that those two competitors are not robust to a small change in the noise level.
\begin{figure}
\setlength\tabcolsep{1pt}
\centering
\begin{tabular}{c}
\includegraphics[width=0.8\textwidth , height=0.2\textwidth]{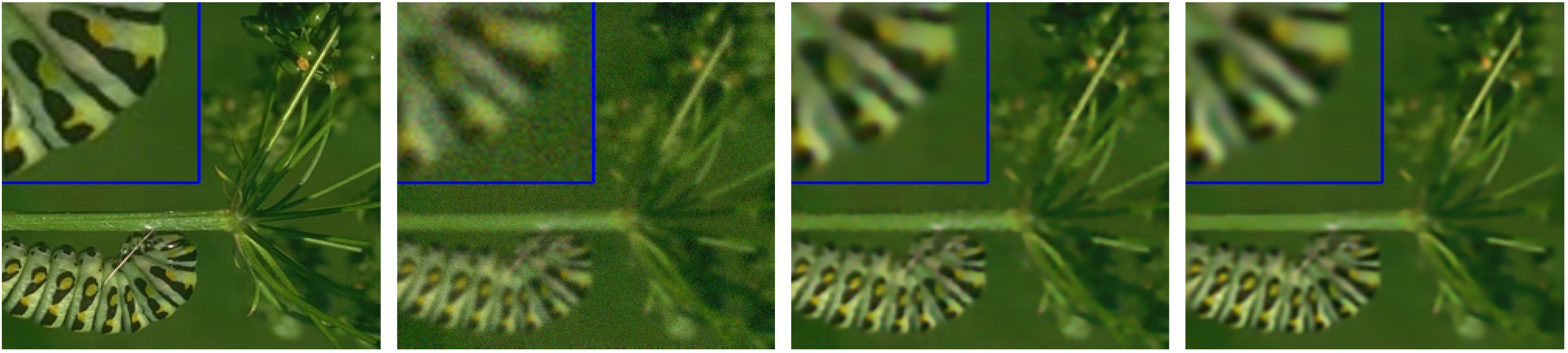}\\
\end{tabular}
\begin{tabular}{cccc}
\hspace{-0.4cm} Ground-truth & \hspace{.3cm} Blurred: 0.509  & \hspace{.4cm} VAR: 0.833 & \hspace{.6cm} EPLL: 0.839 \\
\end{tabular}
\begin{tabular}{c}
\includegraphics[width=0.8\textwidth , height=0.2\textwidth]{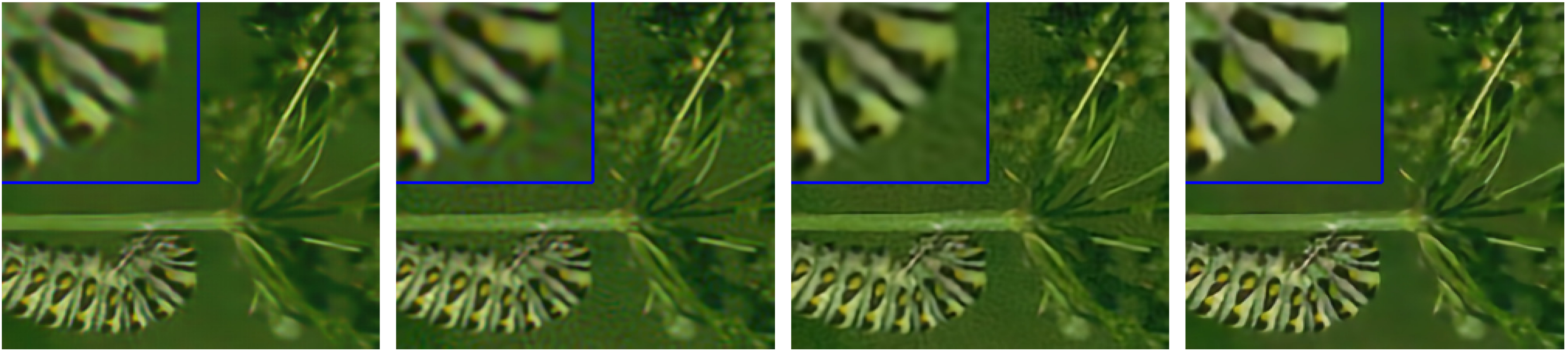}\\
\end{tabular}
\begin{tabular}{cccc}
\hspace{.1cm} MLP: 0.860 &\hspace{.6cm} PDHG: 0.772 & \hspace{.4cm} IRCNN: 0.840  & \hspace{.1cm} iRestNet: \textbf{0.883}\\
\end{tabular}
\caption{Visual results and SSIM obtained with the different methods on one image from the BSD500 test set degraded with GaussianB.}
\label{fig:gaussianB_image}
\end{figure}
\begin{figure}
\setlength\tabcolsep{1pt}
\centering
\begin{tabular}{c}
\includegraphics[width=0.8\textwidth , height=0.2\textwidth]{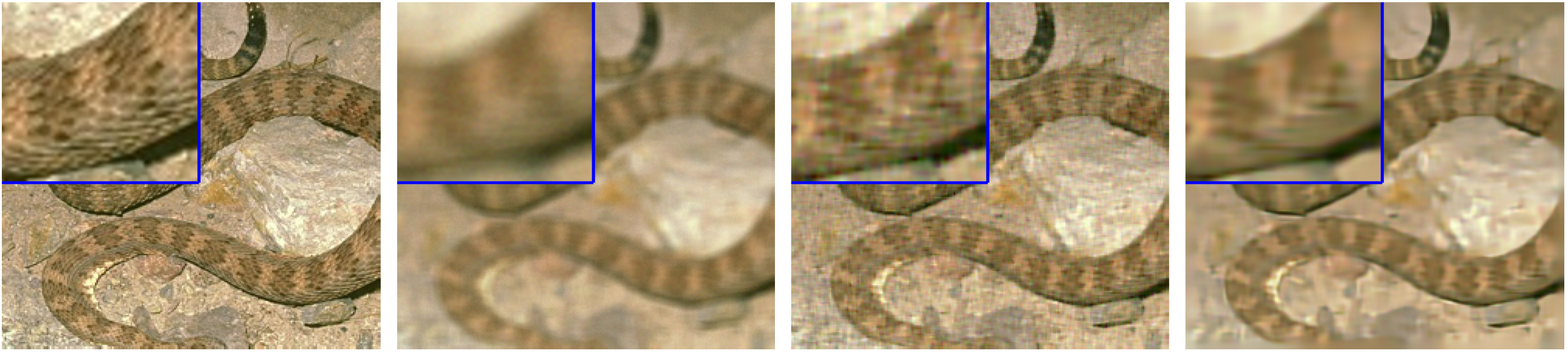}\\
\end{tabular}
\begin{tabular}{cccc}
\hspace{-0.4cm} Ground-truth & \hspace{.3cm} Blurred: 0.344 & \hspace{.4cm} VAR: 0.622 & \hspace{.6cm} EPLL: 0.553 \\
\end{tabular}
\begin{tabular}{c}
\includegraphics[width=0.4\textwidth , height=0.2\textwidth]{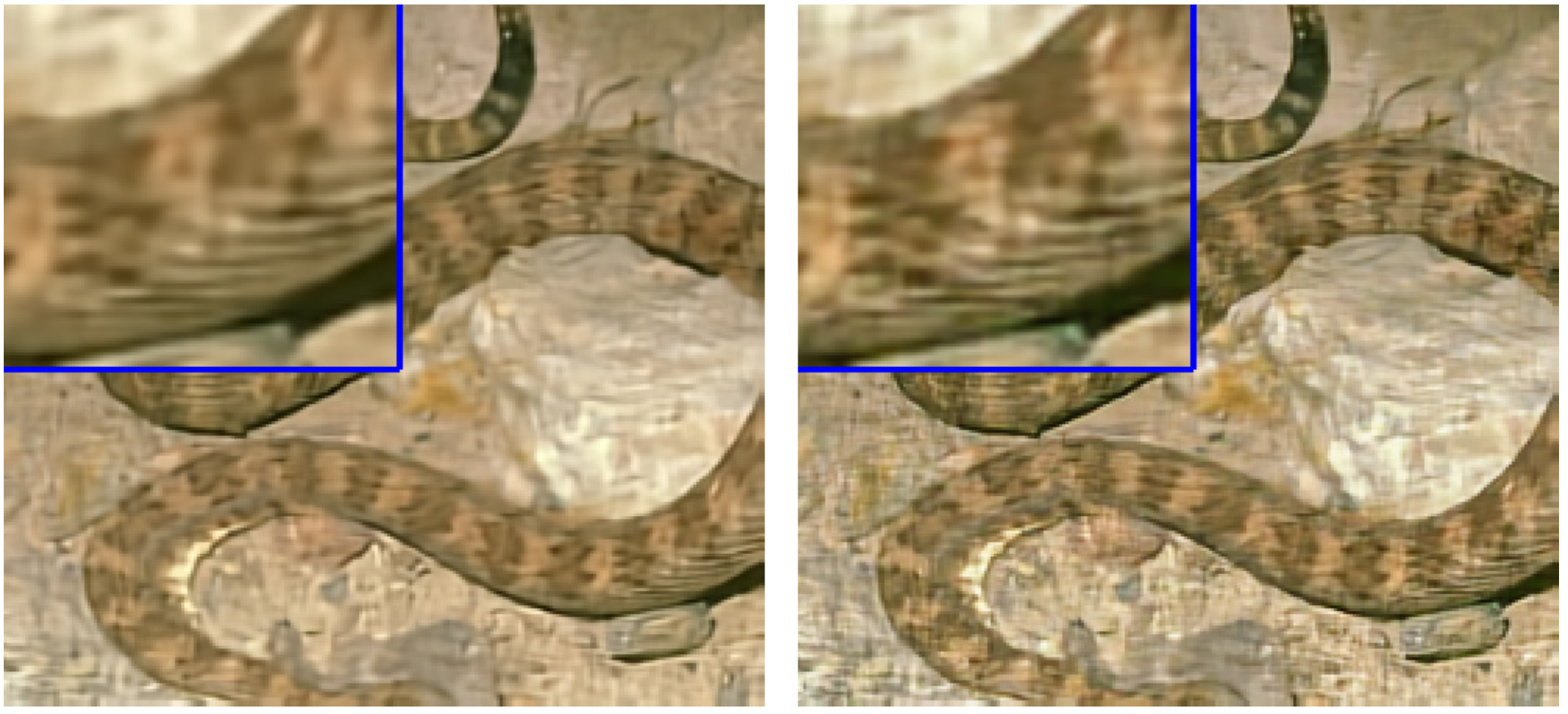}\\
\end{tabular}

\begin{tabular}{cc}
\hspace{.2cm} IRCNN: 0.685 &\hspace{.3cm}iRestNet: \textbf{0.713}\\
\end{tabular}
\caption{Visual results and SSIM obtained with the different methods on one image from the BSD500 test set degraded with Square.}
\label{fig:square_image}
\end{figure}
\begin{figure}
\setlength\tabcolsep{1pt}
\centering
\begin{tabular}{c}
\includegraphics[width=0.8\textwidth , height=0.2\textwidth]{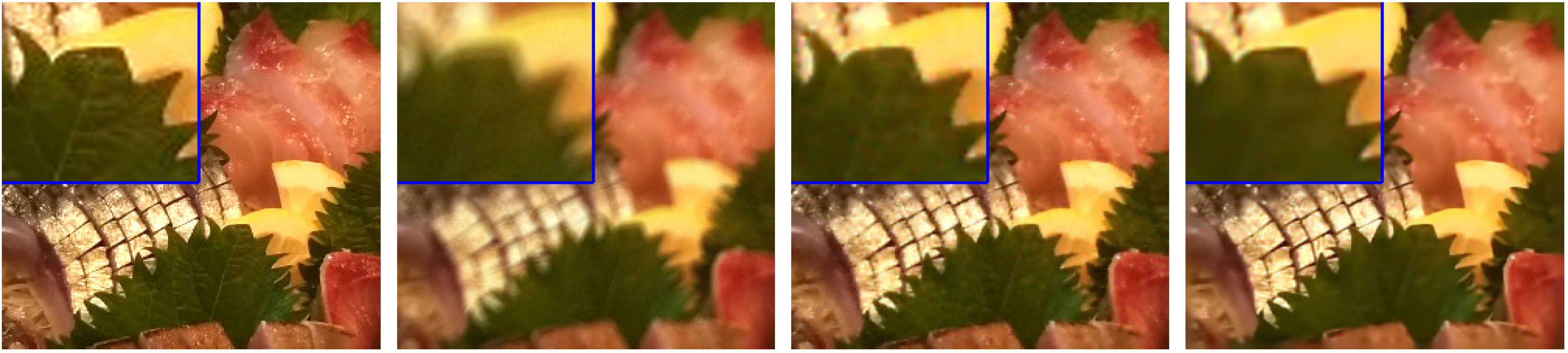}\\
\end{tabular}
\begin{tabular}{cccc}
\hspace{-0.4cm} Ground-truth & \hspace{.3cm} Blurred: 0.576 & \hspace{.4cm} VAR: 0.844 & \hspace{.6cm}  EPLL: 0.849
\\
\end{tabular}
\begin{tabular}{c}
\includegraphics[width=0.6\textwidth , height=0.2\textwidth]{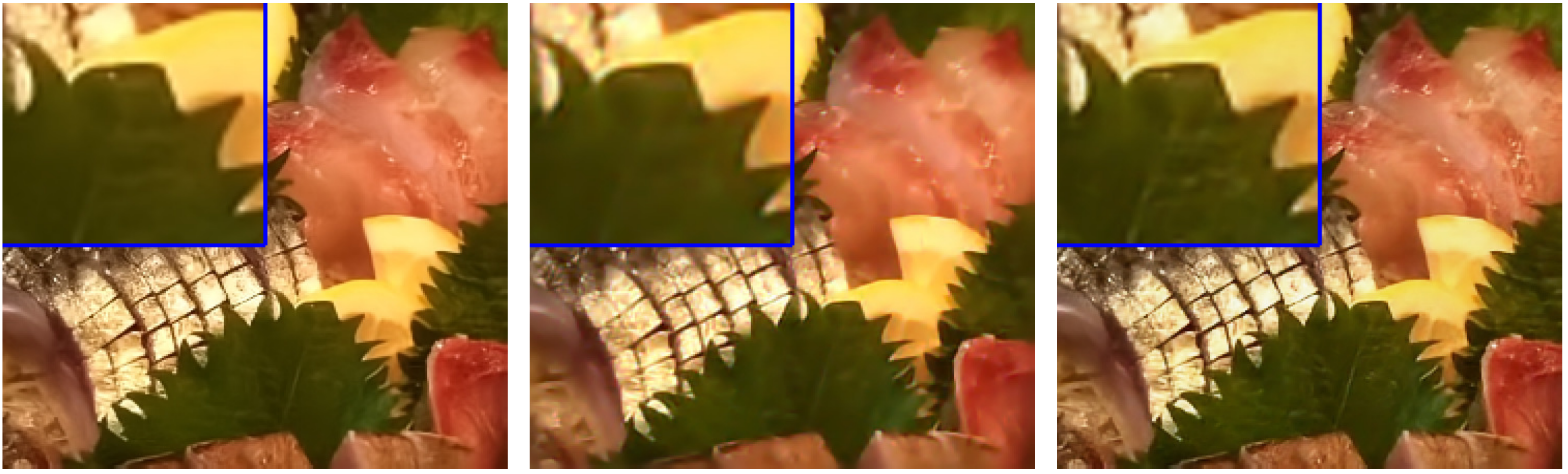}\\
\end{tabular}
\begin{tabular}{ccc}
\hspace{.3cm} IRCNN: 0.906 &\hspace{.4cm} FCNN: 0.856 &\hspace{.2cm} iRestNet: \textbf{0.909}\\
\end{tabular}
\vspace{-.1cm}
\caption{Visual results and SSIM obtained with the different methods on one image from the Flickr30 test set degraded with MotionB.}
\label{fig:motionB_image}
\end{figure}

Figure~\ref{fig:gamma_mu_lambda} shows the stepsize, barrier parameter and regularization weight sequences obtained by passing the image from Figure~\ref{fig:gaussianB_image} through the 40 layers of iRestNet. 
\begin{figure}
\setlength\tabcolsep{1.5pt}
\centering
\begin{tabular}{ccc}
\includegraphics[width=0.32\textwidth , height=0.25\textwidth]{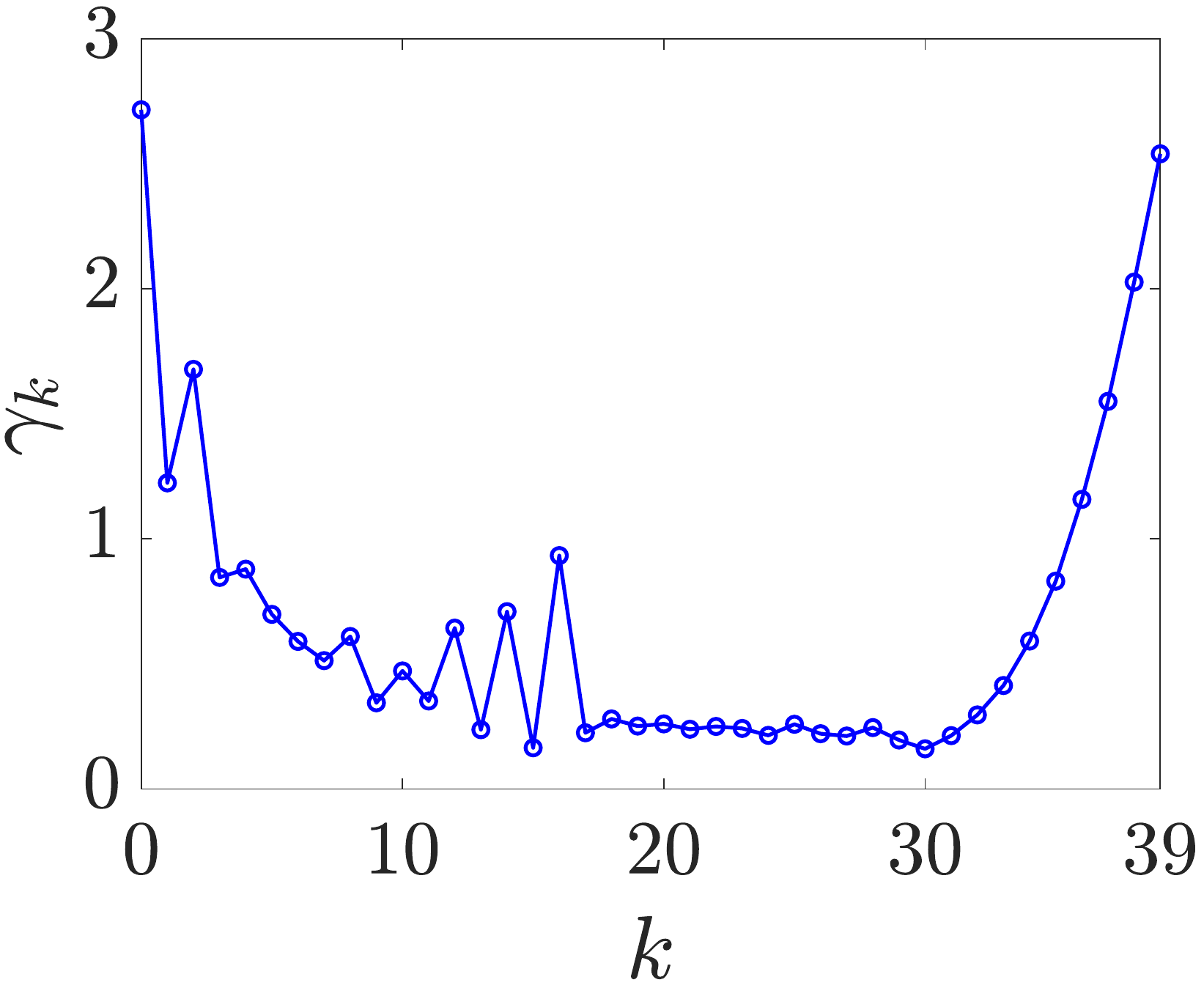}&
\includegraphics[width=0.32\textwidth , height=0.25\textwidth]{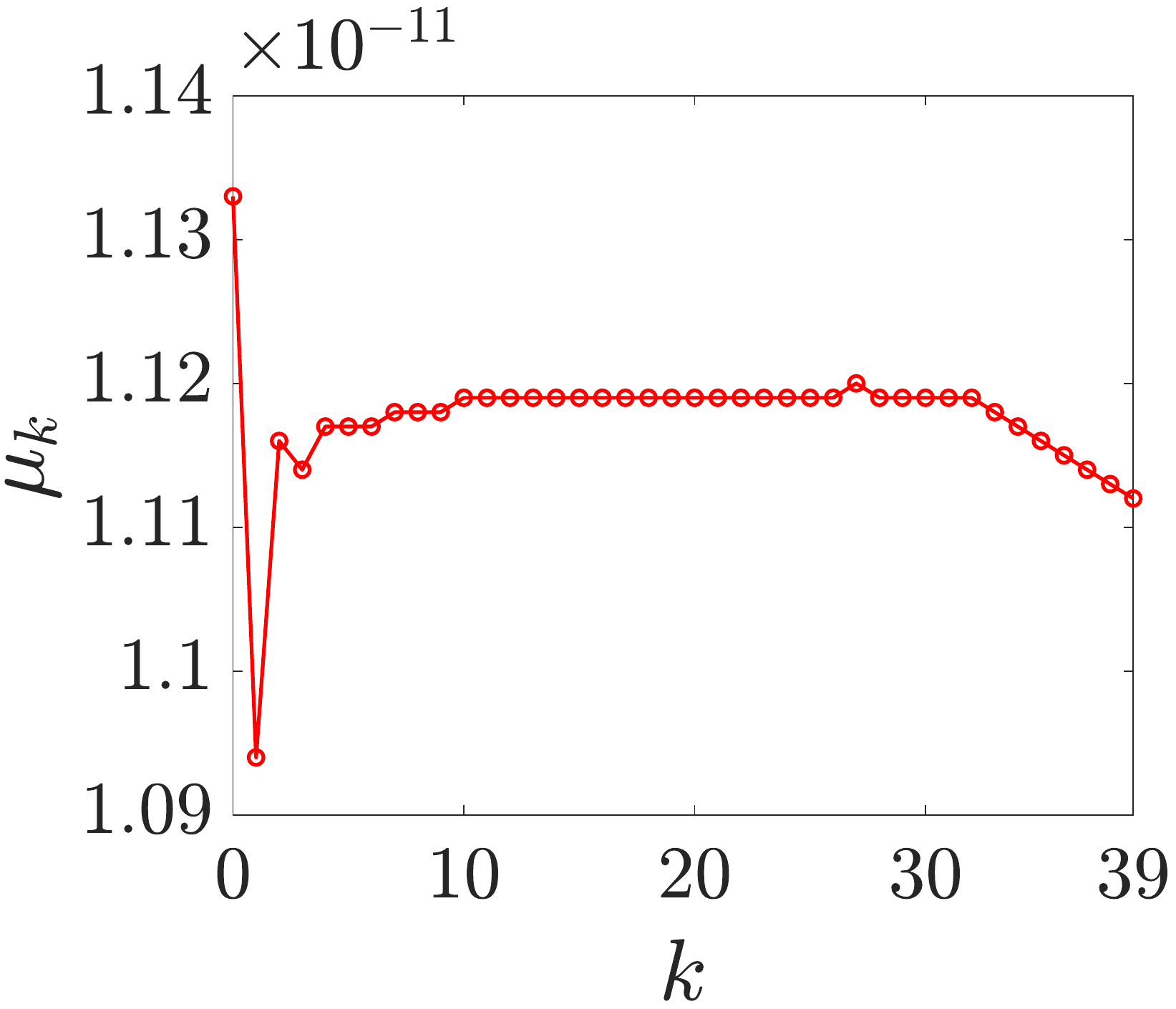}&
\includegraphics[width=0.32\textwidth , height=0.25\textwidth]{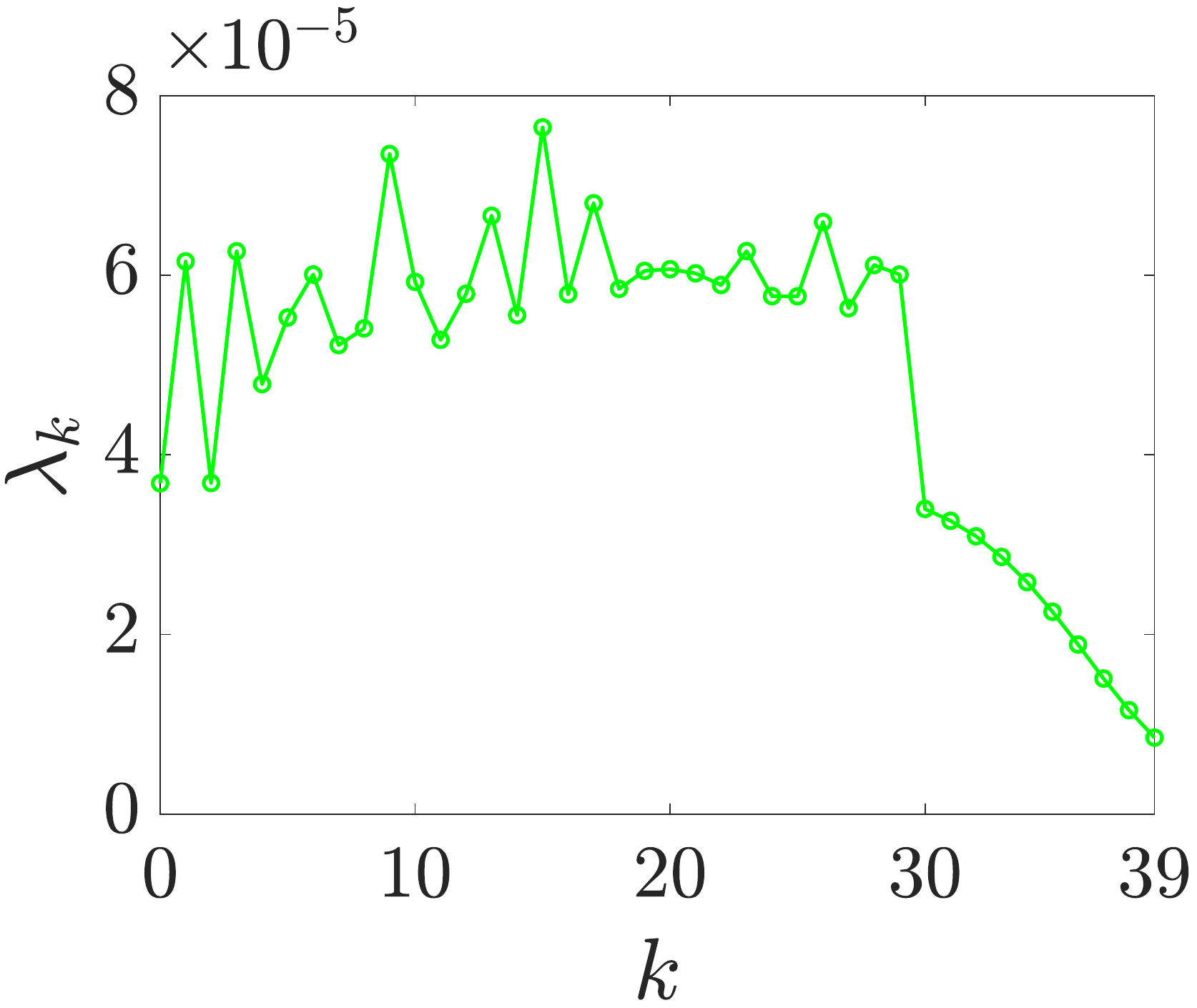}\\
\end{tabular}
\caption{Left to right: estimated {\color{black}stepsize} $(\gamma_k)_{0\leq k\leq K-1}$, {\color{black}barrier parameter} $(\mu_k)_{0\leq k\leq K-1}$ and {\color{black}regularization weight} $(\lambda_k)_{0\leq k\leq K-1}$ for the image from Figure~\ref{fig:gaussianB_image} passed through the network layers.}
\label{fig:gamma_mu_lambda}
\end{figure}

\section{Conclusion}
\label{sec:conclu}

From a variational formulation of an inverse problem, we have derived in this paper a novel neural network architecture by unfolding a proximal interior point algorithm. It can be noted that the proposed approach can be extended to a set of regularization functions, or to penalizations which are parametrized by several variables. Useful constraints on the sought solution can be enforced thanks to a logarithmic barrier, so providing more control over the output of the network. We have shown for three standard types of constraints that the involved proximity operator can easily be computed, and that its derivatives are well-defined and computable. In the case of a quadratic cost function, the theoretical result of Section~\ref{sec:network_stability} regarding the robustness of the network with respect to an input perturbation, ensures the reliability of the proposed method, which is crucial for many applications. It would be interesting to extend the scope of this study to a wider class of problems, and to illustrate this stability result by numerical experiments on different applications like classification. As demonstrated by our experiments in image restoration, iRestNet performs favorably compared to state-of-the-art variational and machine learning methods. An advantage of the proposed approach is that, in contrast with its evaluated competitors, it does not require any knowledge about the noise level and it does not involve any hand-selection of the regularization parameters. One limitation of iRestNet is that the network needs to be trained for a given blur kernel. A direction for future works is to extend the method to situations in which the observation model is not fully known, so as to address blind or semi-blind deconvolution problems.

\ack
We would like to thank the referees and the editor for their valuable comments and suggestions. Marco Prato has received funding from the ECSEL JU programme under the PRYSTINE Project grant agreement n. 783190.

\section*{References}
\bibliographystyle{unsrt}
\bibliography{refs}

\end{document}